\documentclass[a4paper,12pt]{article} 
\pdfoutput=1

\usepackage[margin=1in]{geometry}

\usepackage{amssymb,amsmath,amsthm,txfonts}
\usepackage{hyperref}
\usepackage{ifpdf}
\usepackage{color}
\usepackage{graphicx}
\usepackage{microtype}
\usepackage{paralist}

\newcommand{\useamsrefs}{} 

\ifdefined\useamsrefs
\usepackage[initials,backrefs,msc-links]{amsrefs}
\else
    \newcommand{\cites}{\cite}
\fi

\ifdefined\texmaker
	\usepackage[color]{showkeys} 
\fi

\swapnumbers
\newtheorem{theorem}{Theorem}[section]
\newtheorem{lemma}[theorem]{Lemma}
\newtheorem{corollary}[theorem]{Corollary}
\newtheorem{proposition}[theorem]{Proposition}

\theoremstyle{definition}
\newtheorem{definition}[theorem]{Definition}
\theoremstyle{remark}
\newtheorem{remark}[theorem]{Remark}
\newtheorem{example}[theorem]{Example}

\numberwithin{equation}{section}

\IfFileExists{.git/git.tex}{
\input{.git/git.tex}
\newcommand{\gitversionfootnote}{\blfootnote{Version: \gitauthsdate, \gitshash}}
}
{
\newcommand{\gitversionfootnote}{}
}

\makeatletter
\def\blfootnote{\xdef\@thefnmark{}\@footnotetext} 
\makeatother

\ifpdf
  \graphicspath{{}}
  \newcommand{\fig}[2]{
    \IfFileExists{#1.pdf_tex}{
      \def\svgwidth{#2}\input{#1.pdf_tex}
    }{
      \frame{Missing figure ``#1.pdf''}
      \message{LaTeX Warning: Missing figure ``#1.pdf'' on input line \the\inputlineno}
    }
  }
\else
  \graphicspath{{img/eps/}}
  \newcommand{\fig}[2]{
    \IfFileExists{img/eps/#1.eps_tex}{
      \def\svgwidth{#2}\input{img/eps/#1.eps_tex}
    }{
      \frame{Missing figure ``#1.eps''}
      \message{LaTeX Warning: Missing figure ``#1.eps'' on input line \the\inputlineno}
    }
  }
\fi

\newcommand{\dimension}{n}

\newcommand{\dist}{\operatorname{dist}}

\newcommand{\interior}{\operatorname{int}}
\newcommand{\trace}{\operatorname{trace}}

\newcommand{\esssup}{\operatorname{ess\ sup}\displaylimits}
\newcommand{\essinf}{\operatorname{ess\ inf}\displaylimits}

\newcommand{\grad}{\nabla}
\newcommand{\divo}{\operatorname{div}}

\newcommand{\dx}{\;dx}

\newcommand{\diff}[1]{\;d{#1}}

\newcommand{\ov}[1]{\frac{1}{#1}}

\newcommand{\abs}[1]{\left|#1\right|}
\newcommand{\pth}[1]{\left(#1\right)}
\newcommand{\bra}[1]{\left[#1\right]}
\newcommand{\set}[1]{{\left\{#1\right\}}}
\newcommand{\at}[2]{{{\left.{#1}\right|}_{#2}}}
\newcommand{\ang}[1]{{\left\langle#1\right\rangle}}
\newcommand{\norm}[1]{\left\|#1\right\|}

\newcommand{\cl}[1]{\overline{#1}}	

\newcommand{\al}{\ensuremath{\alpha}}
\newcommand{\be}{\ensuremath{\beta}}
\newcommand{\ga}{\ensuremath{\gamma}}
\newcommand{\de}{\ensuremath{\delta}}
\newcommand{\e}{\ensuremath{\varepsilon}}
\newcommand{\vp}{\ensuremath{\varphi}}
\newcommand{\la}{\ensuremath{\lambda}}
\newcommand{\si}{\ensuremath{\sigma}}
\newcommand{\ta}{\ensuremath{\theta}}

\newcommand{\R}{\ensuremath{\mathbb{R}}}

\newcommand{\Rd}{\ensuremath{{\mathbb{R}^{\dimension}}}}
\newcommand{\Rn}{\Rd}
\newcommand{\Z}{\ensuremath{\mathbb{Z}}}

\newcommand{\T}{\ensuremath{\mathbb{T}}}

\newcommand{\pd}[2]{\frac{\partial {#1}}{\partial {#2}}}

\newcommand{\td}[2]{\frac{d {#1}}{d {#2}}}

\definecolor{grey}{rgb}{0.6,0.6,0.6}

\renewcommand{\labelenumi}{(\alph{enumi})}

\newcommand{\romanlist}{\renewcommand{\labelenumi}{\textup{(}\roman{enumi}\textup{)}}}

\newcommand{\argmax}{\operatorname{arg\,max}\displaylimits}
\newcommand{\argmin}{\operatorname{arg\,min}\displaylimits}

\newcommand{\Tn}{{\T^\dimension}}
\newcommand{\Lip}{\textrm{Lip}}
\DeclareMathOperator{\sign}{sign}
\newcommand{\dom}{\mathcal{D}}
\DeclareMathOperator*{\halflimsup}{\star-limsup}
\DeclareMathOperator*{\halfliminf}{\star-liminf}
\newcommand{\uu}{\underline{u}}
\newcommand{\ou}{\overline{u}}
\newcommand{\parahead}[1]{\bigskip\noindent\textbf{{#1}.}}

\begin{document}

\title{
Periodic total variation flow of non-divergence type in $\Rn$
}
\author{Mi-Ho Giga \and Yoshikazu Giga \and Norbert Po\v{z}\'{a}r}
\date{}
\maketitle

\begin{abstract}
We introduce a new notion of viscosity solutions
for a class of very singular nonlinear parabolic problems
of non-divergence form
in a periodic domain of arbitrary dimension,
whose diffusion on flat parts with zero slope is so strong that
it becomes a nonlocal quantity.
The problems include the classical total variation flow
and a motion of a surface by a crystalline mean curvature.
We establish a comparison principle,
the stability under approximation by regularized parabolic problems,
and an existence theorem for general continuous initial data.
\end{abstract}
\blfootnote{\noindent
M.-H. Giga $\cdot$ Y. Giga $\cdot$ N. Po\v{z}\'{a}r\\         
Graduate School of Mathematical Sciences,
University of Tokyo,
3-8-1 Komaba Meguro-ku,
Tokyo 153-8914, Japan,
\nolinkurl{npozar@ms.u-tokyo.ac.jp}
}
\blfootnote{\emph{Keywords}:
phase transitions, curvature flows, crystalline mean curvature,
total variation flow,
viscosity solutions,
comparison theorems}
\blfootnote{\emph{Mathematics Subject Classification
(2010)}:
35K67, 
35D40,  
35K55,  
35B51,  
35K93  
}
\gitversionfootnote

\section{Introduction}

We are interested in the well-posedness of the initial value problem
for a class of very singular parabolic equations.
They are characterized by the strong diffusive effect on flat
parts of the solution with zero slope.
In fact, the
diffusion is so strong that the equation becomes nonlocal and it
is no longer a classical partial differential equation.
The problems
with such structure have attracted considerable attention due to
their importance in image processing and crystal growth modeling.
A typical example of a problem of divergence type
is the standard total variation flow equation.
Its well-posedness is not trivial
but the theory of monotone operators yields
the unique solvability
of the initial value problem \cite{ACM04}.
However, this theory does not apply if the problem lacks
a divergence structure
and the unique
solvability of the initial value problem has been a long-standing open
problem except in the case of one space variable for which the theory
of viscosity solutions is extendable \cites{GG98ARMA,GG99CPDE}.
Because of the ubiquity of such higher dimensional,
non-divergence form problems
in crystal growth modeling,
it is desirable to find a suitable generalized notion of solutions
and establish the unique solvability of the initial value problem.

Our goal in this paper is to solve this long-standing open problem
in a periodic domain of arbitrary dimension $n \geq 1$ by introducing
a new notion of viscosity solutions.
In full generality, we consider a generalized total variation flow
equation for a function $u(x,t)$ 
on the $n$-dimensional torus $\Omega = \T^n := \Rn/\Z^n$:
\begin{align}
\label{tvf}
    &u_t + F\pth{\nabla u, \divo \pth{\frac{\nabla u}{\abs{\nabla u}}}} = 0
        & &\text{in $Q := \Tn \times (0,T)$,}
\intertext{with the initial condition}
\label{tvf-ic}
    &\at{u}{t=0} = u_0
        && \text{on $\Tn$,}
\end{align}
where $u_t = \pd ut$, $\nabla u = \pth{\pd u{x_1}, \ldots, \pd u{x_n}}$
and $\divo z = \pd {z_1}{x_1} + \cdots + \pd {z_n}{x_n}$.
We assume that $F : \Rn \times \R \to \R$ is a continuous function,
non-increasing in the second variable,
i.e., the operator is \emph{degenerate elliptic}:
\begin{align}
\label{degenerate-ellipticity}
    F(p,\xi) &\leq F(p, \eta)
            && \text{for $\xi,\eta \in \R$, $\xi \geq \eta$, $p \in \Rn$}.
\end{align}
The initial data $u_0$ is assumed to be a continuous (periodic) function
$u_0 \in C(\T^n)$.
We identify functions from $C(\T^n)$ with $\Z^n$-periodic
functions in $C(\Rn)$.

The prototypical example of \eqref{tvf},
with many applications in image processing,
is the standard total variation flow
\begin{align}
\label{standard-tvf}
u_t = \divo \pth{\frac{\nabla u}{\abs{\nabla u}}},
\end{align}
which is formally the $L^2$-gradient flow of the total variation energy
\begin{align*}
 E(\psi) = \int_\Omega \abs{\nabla \psi} \dx.
\end{align*}
The gradient flow structure can be exploited to employ
the abstract theory of monotone operators and obtain well-posedness
for general data in Hilbert space $L^2$;
we refer the reader to \cite{ACM04} for a detailed exposition
and references.
This interpretation will serve
as the basis of evaluating $\divo (\nabla u/\abs{\nabla u})$.
 
The motion of an interfacial surface
by a crystalline mean curvature
is an example of a problem in non-divergence form,
\begin{align}
\label{crystalline-mean-curvature}
u_t = \sqrt{1 + \abs{\nabla u}^2}
    \pth{\divo \pth{\frac{\nabla u}{\abs{\nabla u}}} + c},
\end{align}
where $c$ is a constant forcing.
This is the graph formulation of the anisotropic mean curvature
flow
of a surface
$\Gamma_t = \set{(x, u(x,t)): x \in \Rn} \subset \Rn\times \R$
of codimension one, given as the graph of function $u$,
moving with velocity 
\begin{align}
\label{curvature-flow}
V_\nu = \kappa_\ga + c
\end{align}
in the direction of the unit normal
vector $\nu = (-\nabla u, 1)/\sqrt{1 + \abs{\nabla u}^2}$.
The quantity $\kappa_\ga$ represents the crystalline mean curvature
which is the first variation of the surface energy
\begin{align*}
\int_\Gamma \ga(\nu) \diff{\mathcal{H}^n},
\end{align*}
see \cites{GPR,BNP01a},
where $\mathcal{H}^n$ denotes the $n$-dimensional
Hausdorff measure.
The Finsler metric $\ga(x,u)$ has the particular form
\begin{align*}
\ga(x,u) = \abs{x} + \abs{u},\qquad x\in\Rn, u \in \R,
\end{align*}
and the Wulff shape, which represents the surface
of constant curvature $\kappa_\ga$,
is the unit cylinder
\[W_\ga = \set{(x,u) : \max (\abs{x}, \abs{u}) \leq 1} \subset \R^{n+1}.\]
We mention that $\gamma$ is not purely crystalline in the traditional
sense for $n \geq 2$,
because usually only Finsler metrics with
a polytope as their Wulff shape are considered crystalline;
see \cite{B10} for an introduction to the topic.
We observe that on the flat parts of $u$ with zero slope,
$\kappa_\gamma$ and $\divo (\nabla u/\abs{\nabla u})$ coincide.
Motivated by this example,
we shall call the quantity $\divo (\nabla u/ \abs{\nabla u})$
the \emph{nonlocal curvature},
and we will denote it $\Lambda$ for short.

\parahead{Main result}
The goal of this paper is to introduce
a notion of viscosity solutions of \eqref{tvf},
including \eqref{crystalline-mean-curvature},
and prove its well-posedness as stated in the following theorem.

\begin{theorem}[Main theorem]
Suppose that a continuous function $F : \Rn \times \R \to \R$
is degenerate elliptic in the sense of \eqref{degenerate-ellipticity}.
Then the initial value problem \eqref{tvf}--\eqref{tvf-ic}
with $u_0 \in C(\Tn)$
has a unique global viscosity solution
$u \in C(\T^n \times [0,\infty))$.
If additionally $u_0 \in \Lip(\T^n)$,
i.e., $u_0$ is a periodic Lipschitz function,
then $u(\cdot, t) \in \Lip(\Tn)$ for all $t \geq 0$ and
\begin{align*}
    \norm{\nabla u(\cdot, t)}_\infty &\leq \norm{\nabla u_0}_\infty
    && \text{for $t \geq 0$.}
\end{align*}
\end{theorem}
We also establish a comparison principle for viscosity solutions,
and a stability of solutions under approximation
by regularized problems. Consequently,
we observe that in the case of the standard total variation flow
equation \eqref{standard-tvf} our viscosity solutions
coincide with the semigroup (weak) solutions
given by the theory of monotone operators.

\parahead{Viscosity solutions}
The theory of monotone operators
does not apply to problems that are not of divergence form.
Fortunately,
the problem \eqref{tvf} still has 
a parabolic structure
so that it is expected to have a comparison principle
or an order preserving property.
This structure puts \eqref{tvf} in the scope of the theory of
viscosity solutions.
However, the conventional viscosity theory
for degenerate parabolic problems does not apply
to \eqref{tvf} because the quantity
\[
\Lambda = \divo \pth{\frac{\nabla u}{\abs{\nabla u}}}
    := \frac{1}{\abs q} \trace X\pth{I - \frac{q \otimes q}{\abs{q}^2}}
\]
is unbounded around $q = 0$,
where $q = \nabla u$ and $X = \nabla^2 u$, the Hessian of $u$.
In fact, the singularity at $q =0$ is so strong that
the problem has a nonlocal nature
in the sense that $\Lambda$ is a nonlocal
quantity on the place where $\nabla u = 0$.
This phenomenon still occurs even if \eqref{tvf} is of a divergence form
\cite{FG}
so the equation like \eqref{tvf} is often called a very singular
diffusion equation \cite{GG10}.
If the unboundedness near $\nabla u = 0$
is relatively weak so that the problem is still a local problem
like $p$-Laplace diffusion equation
$u_t - \divo(\abs{\nabla u}^{p-2}\nabla u) =0$
with $1 < p < 2$,
the theory of viscosity solutions has been well-established
\cites{Goto94,IS,OhnumaSato97,Giga06}.
There has been a considerable effort
to generalize the viscosity theory to \eqref{tvf};
however, so far the results have been restricted
to the one-dimensional case
\cites{GG98ARMA,GG01ARMA,GGR11}
or to related level set equations for evolving planar curves
\cites{GG00Gakuto,GG01ARMA};
see also review paper \cite{G04}.

We give a new notion of viscosity solutions
so that the initial value problem is well-posed in higher
dimension at least for periodic initial data.
Following the previous results,
we adjust the class of test functions 
to be compatible with the strong singularity of the operator
at $\nabla u = 0$.
To handle this situation,
we generalize the notion of a \emph{facet},
which will include arbitrary closed sets with nonempty interior.
Then we carefully craft an appropriate class of admissible
\emph{faceted} test functions,
which can be viewed as a natural extension of
the one-dimensional case \cite{GG98ARMA}.
This class consists of Lipschitz functions
whose level set is a facet with a smooth boundary.
The value of the nonlocal curvature
$\divo (\nabla \phi/ \abs{\nabla \phi})$
for such a $\phi$
can be interpreted as
the divergence of a vector field
that is a solution of a certain vector-valued
obstacle problem on the facet.
In contrast to the one-dimensional case,
the nonlocal curvature might be non-constant and even discontinuous
(see \cites{BNP99,BNP01IFB}).
Nevertheless,
we observe that it is only necessary
to evaluate its essential infima and suprema on balls of small radius,
and a pointwise definition of viscosity solutions is possible.
We finally finish the discussion of the construction
of the class of test functions
by pointing out that the operator is a standard differential operator
when the gradient of the solution is nonzero
and the conventional test functions of nonzero gradient can be employed.

The two crucial steps
in establishing a reasonable theory of viscosity solutions
are proving the comparison principle (Theorem~\ref{th:comparison})
and proving the existence result (Theorem~\ref{th:existence}).
The proof of the comparison principle is quite involved
because of the nonlocal nature of the curvature-like quantity
$\Lambda = \divo(\nabla \phi/ \abs{\nabla\phi})$,
which may be discontinuous on facets.
When $\phi_1, \phi_2$ are $C^2$ near $\hat x$ and
$\nabla \phi_1(\hat x) \neq 0$,
the maximum principle asserts that if $\phi_1 \leq \phi_2$
near $\hat x $ and $\phi_1(\hat x) = \phi_2(\hat x)$,
then $\Lambda_1(\hat x)\leq \Lambda_2(\hat x)$
where $\Lambda_i = \divo(\nabla \phi_i/\abs{\nabla\phi_i})$.
This type of monotonicity is necessary for establishing
the standard comparison principle for viscosity solutions \cite{Giga06}.
Our key observation is that this type of monotonicity of $\Lambda$
is still valid on a faceted region (Theorem~\ref{th:monotonicity}).
The proof of the comparison principle is based on a contradiction argument
by assuming the existence of sub- and supersolutions $u$ and $v$,
respectively, satisfying $u\leq v$ at $t = 0$
but $u > v$ at some point for $t \in (0, T)$.
We double the variables with an extra small parameter $\zeta \in \Rn$
and study the maxima of
\begin{align*}
\Phi_\zeta &:= u(x,t) - v(y,s) - \Psi_\zeta &
    &\text{for $x, y \in \T^n$, $t, s \in [0,T]$},\\
\Psi_\zeta &:= \abs{x-y-\zeta}^2/2\e + S,\\
S &:= \abs{t-s}^2/2\sigma + \gamma/(T - t) + \gamma/(T-s),
\end{align*}
where $\e, \si, \ga$ are small positive parameters.
The introduction of $\zeta$ helps flatten the profiles of $u$ and $v$
near a maximizer
$(\hat x, \hat t, \hat x, \hat s)$ of $\Phi_0$.
The history of this device,
which was developed in \cite{GG98ARMA},
goes back to \cite{CGG} and \cite{Goto94}.
Excluding the case where the conventional theory applies,
we may conclude (Corollary~\ref{co:extended-ordering})
that, by fixing small parameters $\e, \si, \ga$,
\begin{align*}
u(x, t)
    &\leq u(\hat x, \hat t) + S(t, \hat s) - S(\hat t, \hat s)
    &&\text{for $x \in V^\la$},\\
v(y, s)
    &\geq v(\hat x, \hat s) + S(\hat t, \hat s) - S(\hat t, s)
    &&\text{for $y \in U^\la$},
\end{align*}
for $t, s \in [0, T]$ for some $\la > 0$,
where
\[
U = \set{x \in \Tn : u(x,\hat t) \geq u(\hat x, \hat t)}
\quad \text{and} \quad 
V = \set{y \in \Tn: v(y, \hat s) \leq v(\hat x, \hat s)},
\]
and $U^\la$ and $V^\la$ are $\la$-neighborhoods of $U$ and $V$,
respectively.
This flattening procedure gives us enough room to construct
the desired admissible test functions for both $u$ and $v$
so that their related nonlocal curvatures are ordered,
which yields a contradiction.

The existence result is proved via stability under approximation
by smooth problems (Theorem~\ref{th:stability})
and by constructing barriers to prevent initial boundary layers
(Lemma~\ref{le:cutoff-wulff-solutions}).
The proof of stability is nontrivial because
of the discrepancy of test functions of the limit problem
and of the smooth problems. 
The key idea is to approximate the spatial profile $f$ of a test function
using the resolvent operator, that is,
formally,
\[
f_a = (I + a \partial E)^{-1} f \qquad \text{for }a > 0,
\]
where $E$ is the total variation energy
and $\partial E$ is its subdifferential.
In our procedure,
we approximate $E$ by a smooth energy $E_m$
to approximate $f_a$ by
\[
f_{a,m} = (I + a \partial E_m)^{-1} f.
\]
The merit of the resolvent operator is that the nonlocal
curvature is well approximated by $(f_{a,m} - f)/a$,
and $f_{a,m}$ is a suitable test function of the smooth problem.
The construction of barriers is also nonstandard but
it is possible to elaborate the idea of \cite{GG99CPDE}.

We are fully aware of the importance of other boundary
conditions like the Dirichlet condition or the Neumann condition.
However, we do not intend to study them in this paper
as these problems seem to need several new ideas.

\parahead{Literature overview}
There is a large amount of literature devoted to the study
of the motion by crystalline mean curvature.
They are classified into at least three approaches--polygonal,
variational and viscosity.

The first approach--a polygonal approach--is to understand
the crystalline curvature $\kappa_\gamma$
for a class of crystal bodies like polygonal shapes.
For two dimensional crystals,
$\kappa_\gamma$ is a constant on facets parallel
to the facets of the Wulff shape $W_\gamma$.
Therefore it is possible to define evolution
of solutions for a class of polygons having such facets when
$\gamma$ is crystalline in strict sense so that its Wulff shape
is a convex polygon.
Such a special family of solutions is called a crystalline flow
or a crystalline motion
introduced in 
\cites{AG89,Taylor91}.
Although there is a higher dimensional approach \cite{GGM},
its validity is limited since
the anisotropic curvature $\kappa_\gamma$ might not be constant
on flat facets \cite{BNP99}
and instant facet breaking or bending may occur.
For a further development of crystalline motion the reader is
referred to \cite{Ishiwata08}.

The second approach--a variational approach--is to understand
$\kappa_\gamma$ as a subdifferential of  the corresponding singular
interfacial energy
when the equation has a divergence structure.
In \cite{FG} it is shown that crystalline motion
can be interpreted as the evolution given by the abstract theory
of monotone operators \cites{Br,Ko}.
This approach enables us to approximate crystalline motion by evolution
by smooth energy and vice versa;
see also \cites{GirK,Gir}
for approximation by crystalline motion.
However, even for one-dimensional problem it is noted
that $\kappa_\gamma + c$ in \eqref{curvature-flow} may not be constant
when $c$ is not spatially constant if the solution
is given by the abstract theory \cite{GG98DS}.
This situation is important since $c$ is often
spatially non-constant in the theory of crystal growth.
In this case if one takes the first approach for planar curve
or graph-like function, polygon is not enough
to describe the phenomena.
However, if one allows to include bent polygons
with free boundaries corresponding to the endpoints of a facet,
it is possible to give a rather explicit solution
\cites{GR08,GR09,GGoR11}.

The situation becomes much more delicate for higher dimensional
crystals even if there is no $c$.
As we mentioned before,
the anisotropic curvature $\kappa_\gamma$ might not be constant
on facets, and instant facet breaking or bending might occur
\cite{BNP99}.
This suggests that one cannot restrict the class of solutions
only to bodies with flat facets like polytope with facets
parallel to those
of the Wulff shape in dimension higher than one
or in the case $c$ is not spatially constant.
A more general theory is necessary.
A notion of generalized solutions and a comparison principle
was established through an approximation by
reaction-diffusion equations in \cites{BN00,BGN00}
for $V_\nu = \gamma \kappa_\gamma$.
However, the existence is known only for convex compact initial data
\cite{BCCN06}.

In general, the curvature $\kappa_\ga$
is known to be only of bounded variation
and bounded on a facet  
\cites{BNP01a,BNP01b}.
One has to solve a nontrivial obstacle problem to calculate
$\kappa_\gamma$ as noted in \cite{GPR}.
The facets where $\kappa_\gamma$ is constant
are called \emph{calibrable} \cite{BNP01IFB}.
The study of this property goes back to a work of
E. Giusti \cite{Giusti78},
where a sufficient condition for calibrability
is given for a two-dimensional convex domain.
Calibrability is related to the study of what is called
Cheeger sets \cite{KL06}.

The third approach is an approach based on the theory of
viscosity solutions that relies on a comparison principle structure
of the problem which we already discussed.
This is the approach taken in this paper. 
The merit is that one can prove existence
and uniqueness in the same class of solutions
for \eqref{tvf}--\eqref{tvf-ic}.
The viscosity approach and the variational approach
for equations of divergence form are compared
in a review paper \cite{GG04}.

There is another approach to definition of solutions of
\eqref{tvf} via an original definition
of composition of multivalued operators
which, while restricted to one dimension at the moment,
allows for the study of facet evolution for quite general
data as well as the regularity of solutions,
including the forcing depending on $x$ and $t$
\cites{MR08,KMR}.

\parahead{Outline of the paper}
We begin the exposition in Section~\ref{sec:nonlocal-curvature}
where we give a rigorous interpretation of
the nonlocal curvature operator $\divo (\nabla u/\abs{\nabla u})$.
This development is then used in Section~\ref{sec:viscosity-solutions},
in which we introduce the new notion of viscosity solution.
We proceed with the proof of the comparison principle
for viscosity solution in Section~\ref{sec:comparison-principle}.
Finally, Section~\ref{sec:existence}
is devoted to the proof of existence of a unique viscosity solution
of the problem via an approximation by regularized parabolic problems.

\section{Nonlocal curvature}
\label{sec:nonlocal-curvature}

In this section we will discuss the interpretation of the term
$\divo \pth{\nabla u / \abs{\nabla u}}$
as the subdifferential of the total variation energy functional.

\subsection{Total variation energy}

The interpretation of the term $\divo \pth{\nabla u/\abs{\nabla u}}$
is motivated by considering
the total variation flow in a smooth bounded domain $\Omega$
in $\mathbf{R}^n$ with the Neumann boundary value condition.
Its formal form is
\begin{align}
\label{tvf-pure}
\begin{cases}
u_t = \divo \pth{\frac{\nabla u}{\abs{\nabla u}}}
    \quad&\text{in}\quad\Omega\times(0,\infty),\\
\pd u\nu=0
    &\textup{on}\quad\partial\Omega\times(0,\infty),\\
\at{u}{t=0} = u_0 &\text{on}\quad\Omega,
\end{cases}
\end{align}
where $\nu=\nu_{\partial\Omega}$ denotes
the unit exterior normal of $\partial\Omega$
so that $\pd u\nu$ is the exterior normal derivative of $u$.
Since the equation \eqref{tvf-pure} is singular at $\nabla u=0$,
it is nontrivial to formulate a solution in a rigorous way.
Fortunately, the classical theory of monotone operators
applied to the subdifferential of a total energy gives
a rigorous interpretation \cite{ACM04}.
Let us briefly recall its theory.
We consider the total variation energy $E = E_\Omega$,
\begin{align*}
E(\psi)=
\begin{cases}
\int_\Omega\abs{\nabla\psi}
    &\text{if}\quad\psi\in BV(\Omega)\cap L^2(\Omega), \\
\infty
    &\text{otherwise,}
\end{cases}
\end{align*}
in the Hilbert space $H=L^{2}(\Omega)$.
The energy $E$ is convex and lower semi-continuous on $H$
and evidently $E\not\equiv\infty$.
Hence a general theory of monotone operators,
initiated by
Y. K\=omura \cite{Ko} and developed by H. Br\'{e}zis \cite{Br}
and others,
guarantees that
there is a unique solution $\psi:[0,\infty)\to H$
which is continuous up to $t=0$
and absolutely continuous in $[\delta,T]$ for any $T>\delta>0$ such that
\begin{equation}
\td u t\in -\partial E(u(\cdot, t))
\quad\text{for a.e.}\quad t>0,\quad\text{and}\quad u|_{t=0}=u_0 \label{2.4}
\end{equation}
for a given arbitrary $u_0\in H$.
Here $\partial E$ denotes the subdifferential of $E$, i.e.,
\begin{equation*}
\partial E(\psi)=\bigl\{w\in H\mid E(\psi+h)-E(\psi)\geq\langle w,h\rangle\textup{ for all }h\in H\bigr\},
\end{equation*}
where $\ang{\ ,\ }$ is the inner product of $H$.
In our setting it is the $L^2$-inner product
$\ang{w,h}=\int_{\Omega}wh\dx$.
It can be shown that $\partial E(\psi)$ is a convex,
closed set and in general it can be empty depending on $\psi\in H$.
We introduce the domain of $\partial E$
\begin{align*}
    \mathcal{D}(\partial E)
        = \set{\psi \in H: \partial E(\psi) \neq \emptyset}.
\end{align*}
 The equation (\ref{2.4}) is a rigorous form of \eqref{tvf-pure}.
Note that the solution $u$ is unique although
the equation (\ref{2.4}) may seem ambiguous.
It turns out that the solution $u$ is right differentiable
for all $t>0$ and its right derivative $\textup{d}^{+}u/\textup{d}t$
agrees with the negative of the \emph{minimal section}
(\emph{canonical restriction})
$\partial^{0}E(u(t))$ of $\partial E(u(t))$.
In other words,
\begin{equation*}
\frac{\textup{d}^{+}u}{\textup{d}t}
    =-\partial^{0}E(u(\cdot, t))
        \quad\textup{for all}\quad t>0,
\end{equation*}
where
\begin{equation*}
\partial^{0}E(\psi) :=
    \argmin \set{\norm{w}_H \;\bigm\vert\; w \in \partial E(\psi)}.
\end{equation*}
Since $\partial E(\psi)$ is closed convex,
$\partial^{0}E(\psi)$ is uniquely determined.
There is a characterization of the subdifferentials
which can be found
in \cite[Lemma~2.4]{ACM04} or \cite{AD}.
It states that $w\in\partial E(\psi)$
for $\psi \in L^2(\Omega)$
if and only if $\psi \in BV(\Omega)$
and there exists $z\in L^{\infty}(\Omega; \Rn)$ such that
\begin{subequations}
\label{subdiff-domain-conditions}
\begin{gather}
w=-\textup{div}\,z\in L^{2}(\Omega), \label{2.5} \\
\norm{z}_\infty \leq 1 \quad \text{and}\quad
\ang{w, \psi} = E(\psi),
 \label{2.6} \\
z\cdot\nu_{\partial\Omega}=0\quad\textup{on}\quad\partial\Omega\quad\textup{(in }H^{-1/2}\textup{ sense),} \label{2.7}
\end{gather}
\end{subequations}
where $\nu_{\partial\Omega}$ is the unit outer normal vector to $\partial\Omega$.
In particular,
if $\psi \in \Lip(\Omega)$ then
$\nabla\psi \in L^\infty(\Omega)$
and the condition \eqref{2.6} is equivalent to
$z(x) \in \partial W(\nabla\psi(x))$ for a.e. $x \in \Omega$,
where $\partial W$ is the subdifferential of $W(p) = \abs{p}$.
It is given as
\begin{align*}
    \partial W(p) =
    \begin{cases}
        \cl B_1(0) & p = 0,\\
        \set{\frac{p}{\abs{p}}} & p \neq 0,
    \end{cases}
\end{align*}
where $B_r(x)$ denotes the open ball of radius $r$ centered at $x$.
We observe $z(x)=\nabla \psi(x)/\abs{\nabla \psi(x)}$
whenever $\nabla\psi(x)\neq 0$.

To motivate the following construction,
let us heuristically study the canonical restriction $\partial^{0}E(\psi)$
for a simple profile.
Assume that $\psi$ takes its maximum $0$ in a closure $\cl\Omega_-$
of a domain $\Omega_-$ with the property that
$\cl{\Omega}_-\subset\Omega$
and $\psi<0$ in $\Omega\backslash\cl{\Omega}_-$.
Assume that $z$ satisfying (\ref{2.5})--(\ref{2.7}) exists
and $z=\nabla\psi(x)\big/\bigl|\nabla\psi(x)\bigr|$
in $\Omega\backslash\bar{\Omega}_-$.
Since $\psi$ attains its maximum on $\Omega_-$,
$z\cdot\nu_{\partial\Omega_-}=-1$ on $\partial\Omega_-$,
at least formally, where $\nu_{\partial\Omega_-}$ is
the unit exterior normal of $\partial\Omega_-$,
and we observe that
\begin{align*}
- \partial^0 E(\psi) =
\begin{cases}
    \divo z_0 & \text{in $\Omega_-$},\\
    \divo \frac{\nabla \psi}{\abs{\nabla \psi}} 
        &\text{in $\Omega \setminus \cl \Omega_-$},
\end{cases}
\end{align*}
such that $\divo z_0$ minimizes the variational problem
\begin{equation}
\min\set{\int_{\Omega_-}|\divo z|^2\dx
    \biggm|
    z\cdot\nu_{\partial\Omega_-}=-1
    \text{ on }\partial\Omega_-\textup{ and }
    |z|\leq 1\text{ a.e. in }\Omega_-}. \label{2.8}
\end{equation}
The quantity $\divo \pth{\nabla \psi/\abs{\nabla\psi}}$
depends on the second derivative of $\psi$
and it is nonpositive if $\psi$ is concave.
This quantity is formally the one-Laplacian of $\psi$.
Since it is not a local quantity on $\Omega_-$,
we rather would like to call $\textup{div}\,z_0$
a nonlocal one-Laplacian of $\psi$ on $\Omega_-$.
Geometrically speaking, it corresponds to
a nonlocal mean-curvature-like quantity.
In this paper we shall call this quantity a nonlocal
curvature for short.
The variational problem (\ref{2.8}) is still convex
as it admits at most one minimizer $\textup{div}\,z$,
although there are several degrees of freedom in the choice of $z_0$.
Note that $\divo z_0$ does not depend on the choice $\Omega$
provided that $\cl\Omega_- \subset \Omega$.
In the literature the vector field $z$ satisfying $z\cdot\nu_{\partial\Omega_-}=-1$
on $\partial\Omega_-$ and $|z|\leq 1$
is often called a Cahn-Hoffman vector field.

For our purpose,
we have to consider the total variation energy
for a periodic function.
A periodic function in $\Rn$ is a function satisfying
$\phi(x+e_i) = \psi(x)$ for all $x \in \Rn$
and some orthogonal basis $\set{e_k}_{k=1}^n$ of $\Rn$
(independent of $x$).
It is identified with a function on the $n$-dimensional torus
$\T^n = \R^n/\Gamma$
where $\Gamma$ is the lattice $\Z e_1 + \cdots + \Z e_n$.
The set $\T^n$
is interpreted as the set of equivalence classes $\set{x + \Gamma: x\in\Rn}$
with the induced metric,
topology, and absolute value, namely
\begin{align}
\label{dist-on-torus}
\dist (x, y) &= \dist_{\Rn} (x+\Gamma, y +\Gamma), &
\abs{x} := \dist(y, 0) = \inf_{\xi \in \Gamma} \abs{x + \xi}_{\Rn},
\quad x, y \in \T^n.
\end{align}
In particular, a ball $B_r(x)$ is defined as
$B_r(x) := \set{y \in \T^n: \abs{x - y} < r}$.
To simplify the exposition,
we set $\Gamma = \Z^n$ throughout the paper.

We consider the total variation energy $E = E_{\T^n}$
for periodic $L^2$ functions.
It is of the form
\begin{align*}
E_\Tn(\psi) =
\begin{cases}
\int_\Tn \abs{\nabla \psi} & \text{if } \psi\in BV(\Tn) \cap L^2(\Tn),\\
\infty & \text{otherwise},
\end{cases}
\end{align*}
in the Hilbert space $H=L^2(\Tn)$.
The unique solvability of \eqref{2.4} as well as the characterization of
the speed is still valid for $E = E_\Tn$
with obvious modification.
For example,
\eqref{2.7} is unnecessary.
The characterization \eqref{2.8}
remains unchanged provided that $\cl\Omega_-$ is contained
in a fundamental domain of $\Tn$.
In this case,
$-\partial^0 E_\Tn(\psi)$ agrees with $-\partial^0 E(\psi)$
in $\Omega_-$.
In what follows, we always take $E = E_\Tn$ instead of $E_\Omega$.

\subsection{Resolvent equation}
\label{sec:resolvent-equation}

It is possible to approximate the evolution of 
\eqref{tvf-pure}
via an implicit Euler scheme
with a time step $a > 0$
that has the form of
the following resolvent problem on $\Tn$:
for given $a > 0$ and $f \in L^2(\Tn)$ find $u \in L^2(\Tn)$ such that
\begin{align}
\label{resolvent-problem}
    u + a\partial E(u) \ni f.
\end{align}

Due to the standard theory of monotone operators,
this problem has a unique solution $u \in L^2(\Tn)$,
i.e.,
for every $f \in L^2(\T^n)$ there exist
unique $u \in \mathcal{D}(\partial E)$
and $v \in \partial E(u)$
such that $u + a v = f$;
see for instance \cite[\S9.6.1, Theorem 1]{Evans}.

The following comparison theorem was proved in \cite{CasellesChambolle06}.

\begin{theorem}
\label{th:resolvent-comparison}
Let $u_1, u_2 \in L^2(\Tn)$ be two solutions of
\eqref{resolvent-problem}
with right-hand sides $f_1, f_2 \in L^\infty(\Tn)$, respectively.
If $f_1 \leq f_2$, then $u_1 \leq u_2$.
\end{theorem}

\begin{proof}
We realize that the solution $u \in L^2(\Tn)$
of \eqref{resolvent-problem}
provided by the theory of monotone operators
is a solution in the sense of \cite[Definition 2]{CasellesChambolle06},
and therefore \cite[Theorem 2]{CasellesChambolle06} applies.
\end{proof}

Because $\partial E$ is a maximal monotone operator,
the resolvent problem \eqref{resolvent-problem} offers
an alternative to finding the minimal section element of
the subdifferential $\partial E$.

\begin{proposition}
\label{pr:resolvent-convergence}
Whenever $f \in \mathcal{D}(\partial E)$,
we have
\begin{align*}
\frac{f_a - f}{a} &\to -\partial^0 E(f)
    &&\text{in $L^2(\Tn)$ as $a\to 0$},
\end{align*}
where $f_a$ is the solution $u$ of \eqref{resolvent-problem} for
given $a > 0$.
\end{proposition}

\begin{proof}
See \cite[Theorem~3.56]{Attouch} (also \cite[\S9.6.1, Theorem~2]{Evans}).
\qedhere\end{proof}

\subsection{Nonlocal mean curvature of a facet}

To properly capture the evolution of ``non-convex'' facets,
we need to investigate a more general profile
than the one considered in \eqref{2.8}.
A pair of open sets $(\Omega_-, \Omega_+)$ is all that is needed to
describe a general facet $D = \cl\Omega_- \setminus \Omega_+$.
Indeed, it will be shown below
(Proposition~\ref{pr:support-function}) that the quantity
$\divo (\nabla \psi / \abs{\nabla \psi})$
on the facet $D$ of $\psi$
depends only on the sign of $\psi$ outside of the facet $D$.
This motivates the following two definitions.

\begin{definition}
\label{def:smooth-pair}
Let $\Omega_-, \Omega_+ \subset \Tn$ be a pair of open sets
such that $\Omega_+ \subset \Omega_-$
and $\Omega_- \setminus \cl \Omega_+ \neq \emptyset$.
We say that $(\Omega_-, \Omega_+)$ is a \emph{smooth pair}
if
$\Omega_\pm$ have $C^2$ (possibly empty) boundaries
and $\dist(\partial \Omega_-, \partial \Omega_+) > 0$.
This includes the cases when $\Omega_+ = \emptyset$
or $\Omega_- = \Tn$, i.e.,
when one or both of $\partial \Omega_\pm$ is empty
and
the distance from an empty set is $+\infty$ by definition.
\end{definition}

\begin{definition}
\label{def:support-function}
Let $\Omega_\pm$ be a pair of open sets,
$\Omega_+ \subset \Omega_- \subset \Tn$,
and
$\psi \in \Lip(\Tn)$.
We say that $\psi$ is a \emph{support function of $(\Omega_-, \Omega_+)$}
if
\begin{align}
\label{facet-function}
    \psi(x) \quad
    \begin{cases}
    > 0 & \text{in $\Omega_+$},\\
    = 0 & \text{on $D := \cl\Omega_- \setminus \Omega_+$},\\
    < 0 & \text{in $\cl\Omega_-^c$},
    \end{cases}
\end{align}
where $\cl\Omega_-^c := \Tn \setminus \cl\Omega_-$
denotes the complement of the closure of $\Omega_-$.
\end{definition}

\begin{remark}
\label{re:support-function-symmetry}
It is easy to see that if $\psi$ is a support function of
a smooth pair $(\Omega_-, \Omega_+)$
then $-\psi$ is a support function of the smooth pair
$(\cl\Omega_+^c, \cl\Omega_-^c)$.
\end{remark}

For an open set $U \subset \Tn$
we define \cite{Anz}
\begin{align*}
X_2(U) := \set{z \in L^\infty(U; \Rn) : \divo z \in L^2(U)}.
\end{align*} 
If $\partial U$ is Lipschitz continuous, the normal trace $z\cdot\nu_{\partial U}$ is well-defined
in the sense of $H^{-1/2}(\partial\Omega)$ for every $z \in X_2(U)$,
see e.g. \cite{FM}.
However, here we rather use an equivalent definition based
on Green's formula.
The following version of Green's theorem for vector fields in $X_2(U)$
was proven in \cite{Anz}
for
bounded open sets $U \subset \Rn$,
but the modification for $\Tn$ is straightforward.

\begin{theorem}
\label{th:green}
Let $U \subset \Tn$ be an open set
with Lipschitz boundary and let $z \in X_2(U)$.
Then $z \cdot \nu_{\partial U} \in L^\infty(\partial U)$
such that $\norm{z \cdot \nu_{\partial U}}_{L^\infty(\partial U)} \leq \norm{z}_{L^\infty(U)}$, and
\begin{align*}
\int_U u \divo z + \int_U z \cdot \nabla u
= \int_{\partial U} u(z \cdot \nu_{\partial U}) \diff{\mathcal{H}^{n-1}}
\qquad \text{for all } u \in \Lip(U).
\end{align*}
\end{theorem}

\begin{definition}
Let $(\Omega_-, \Omega_+)$ be a smooth pair.
We say that a vector field
$z \in X_2(\interior D)$ for $D = \cl \Omega_- \setminus \Omega_+$
is a \emph{Cahn-Hoffman vector field in $D$}
if $z$ satisfies
\begin{equation}
\abs{z} \leq 1 \quad \text{a.e. in } D \label{2.9}
\end{equation}
as well as the boundary condition
\begin{equation}
z\cdot\nu_{\partial\Omega_-}=-1\quad\text{on}\quad\partial\Omega_-,
\qquad
z\cdot(-\nu_{\partial\Omega_+})=1\quad\textup{on}\quad\partial\Omega_+, \label{2.10}
\end{equation}
a.e. with respect to $\mathcal{H}^{n-1}$.
\end{definition}

The following proposition shows that the boundary condition
\eqref{2.10} is natural
and indeed holds for any vector field $z \in X_2(\Tn)$
that satisfies \eqref{2.6}
for a given $\psi \in \Lip(\Tn)$
with a sufficiently smooth level-set.

\begin{proposition}
\label{pr:trace}
Suppose that $u \in \Lip(\Tn)$
and define $G_t := \set{x \in \Tn: u(x) > t}$.
Suppose that $\partial G_0 \in C^1$
and that there exists $z \in X_2(\Tn)$
such that $z \in \partial W(\nabla u)$ a.e. on $\Tn$.
Then $z\cdot\nu_{\partial G_0} = -1$
a.e. on $\partial G_0$ with respect to $\mathcal{H}^{n-1}$.
\end{proposition}

\begin{proof}
Since $\norm{z}_\infty \leq 1$,
Theorem~\ref{th:green}
implies that $z\cdot\nu_{\partial G_0} \geq -1$
a.e. with respect to $\mathcal{H}^{n-1}$
and
\begin{align}
\label{trivial}
\int_{G_0} \divo z = \int_{\partial G_0} z\cdot\nu_{\partial G_0} \diff{\mathcal{H}^{n-1}}
\geq -\mathcal{H}^{n-1}(\partial G_0).
\end{align}
Hence it is enough to prove 
\begin{align}
\label{keyineq}
\int_{G_0} \divo z \leq
- \mathcal{H}^{n-1}(\partial G_0).
\end{align}
We define the cutoffs
\[
u_k = \min(\max(u,0), k), \qquad k \geq 0.
\]
Clearly
\[
\nabla u_k = \begin{cases}
\nabla u(x) & \text{a.e. }x, 0 < u(x) < k,\\
0 & \text{a.e. otherwise}.
\end{cases}
\]
In particular, $z \in \partial W(\nabla u_k)$ a.e. and
therefore
\[
\int_\Tn z \cdot \nabla u_k = \int_\Tn \abs{\nabla u_k}.
\]
Theorem~\ref{th:green} on the left and coarea formula for BV function $u_k$ on the right
imply
\begin{align}
\label{coarea}
-\int_\Tn u_k \divo z =
\int_\Tn z \cdot \nabla u_k = \int_\Tn \abs{\nabla u_k}
 = \int_0^k \int_\Tn \abs{\nabla \chi_{G_t}} \dx \diff t,
\end{align}
where $\chi_{G_t}$ is the characteristic function of $G_t$.
Using Fubini's theorem, we rewrite the first term as
\begin{align}
\label{fubini}
\begin{aligned}
\int_\Tn u_k(x) \divo z(x) \dx &=  \int_\Tn \int_0^{u_k(x)} \divo z(x) \diff{t} \dx\\
&= \int_0^k \int_{G_t} \divo z(x) \dx \diff{t}.
\end{aligned}
\end{align}
Since \eqref{coarea} and \eqref{fubini} hold for all $k \geq 0$,
we conclude that
\begin{align}
\label{div-per}
\int_{G_t} \divo z
= - \int_\Tn \abs{\nabla \chi_{G_t}}
\qquad \text{a.e. } t \geq 0.
\end{align}

Clearly $G_t \nearrow G_0$ as $t \searrow 0$
and therefore the dominated convergence theorem implies
\[
\int_{G_0} \divo z = \lim_{t\to0+} \int_{G_t} \divo z,
\]
and $\chi_{G_t} \to \chi_{G_0}$ in $L^1(\Tn)$.
Let $t_m$ be a sequence of $t$ for which \eqref{div-per} holds
and $t_m \to 0$ as $m \to \infty$.
The lower semi-continuity of total variation implies
\[
\int_{G_0} \divo z= \lim_{m\to\infty} \int_{G_{t_m}} \divo z
    = - \lim_{m\to\infty} \int_\Tn \abs{\nabla \chi_{G_{t_m}}}
    \leq - \int_{G_0} \abs{\nabla \chi_{G_0}}
    = - \mathcal{H}^{n-1} (\partial G_0),
\]
where the last equality holds since $\partial G_0 \in C^1$.
Thus we have established \eqref{keyineq} and the lemma follows.
\qedhere
\end{proof}

We consider the variational problem
\begin{equation}
 \label{vp-facet}
\min\left\{\int_{D}|w|^{2}\textup{d}x\biggm|w=\textup{div}\,z\textup{ and }z\textup{ is a Cahn-Hoffman vector field in }D\,\right\}.
\end{equation}
Of course, (\ref{2.8}) corresponds to the case when $\Omega_+$
is an empty set.
Since (\ref{2.9}) and (\ref{2.10}) are fulfilled
for convex combinations of Cahn-Hoffman vector fields as well,
the problem \eqref{vp-facet} is a convex (obstacle) problem.
Moreover, it is lower semi-continuous in $L^{2}(D)$ with respect to $w$
(the set of Cahn-Hoffman vector fields is closed in $L^2(D)$).
Thus there exists at least one minimizer $w_0$ provided that
there is a Cahn-Hoffman vector field.
Fortunately, if the pair $(\Omega_-, \Omega_+)$ representing the facet
is smooth, such a vector field exists.
Let us recall that $w \in -\partial E(\psi)$ for $\psi \in \Lip(\Tn)$
if and only if there exists $z \in X_2(\Tn)$ such that
\begin{align}
 \label{lip-subdiff}
 \divo z = w \qquad \text{and} \qquad 
z \in \partial W(\nabla \psi) \qquad \text{a.e. on } \Tn.
\end{align} 

\begin{proposition}
\label{pr:support-function}
Let $(\Omega_-, \Omega_+)$ be a smooth pair.
Then there exists a support function $\psi$ of $(\Omega_-, \Omega_+)$
in the sense of Definition~\ref{def:support-function}
such that
$\psi \in \mathcal{D}(\partial E)$,
and a Cahn-Hoffman vector field in
$D = \cl \Omega_- \setminus \Omega_+$.
Consequently,
there exists $w_0$, the unique solution of the variational problem
\eqref{vp-facet}, and
$w_0(x) = - \partial^0 E(\psi)(x)$ a.e. in $D$.
In fact, the value of $-\partial^0 E(\psi)$ on $D$
is independent of the choice of a support function $\psi$
of $(\Omega_-,\Omega_+)$ such that
$\psi \in \mathcal{D}(\partial E)$.
\end{proposition}

\begin{proof}
For a given nonempty set $E \subset \Tn$,
let us define the signed distance function
\begin{align*}
    d_E(x) = \dist(x, E) - \dist(x, E^c).
\end{align*}

Since $(\Omega_-, \Omega_+)$ is a smooth pair,
$\de_0 := \dist(\partial \Omega_-, \partial \Omega_+) > 0$.
Moreover,
the boundaries $\partial \Omega_\pm$ are compact $C^2$ surfaces
and therefore they are sets of \emph{positive reach}:
there exist positive constants $\de_-$, $\de_+$,
such that $d_{\Omega_\pm} \in C^2$ in
$\set{x : \dist(x, \partial \Omega_\pm) < \de_\pm}$;
see \cites{DZ11,KP}.

We set $\de := \frac13 \min (\de_-, \de_+, \de_0) > 0$.
Let $\ta \in C^\infty(\R)$ be a smooth nondecreasing function such that
\begin{align*}
0 &\leq \ta'(\si) \leq 1 & &\text{for all $\si \in \R$}, &
\ta(\si) &= 
\begin{cases}
\si &\text{for $\abs{\si} \leq \de/2$},\\
 \frac34\de \sign \si & \text{for $\abs\si \geq \de$}.
\end{cases}
\end{align*}

We also define $\zeta \in \Lip(\R)$,
\begin{align*}
\zeta(\si) := \max(\min(\si, \frac12\de), -\frac12\de) = 
\begin{cases}
\si & \text{for $\abs\si \leq \de/2$},\\
\frac12 \de \sign \si & \text{otherwise.} 
\end{cases}
\end{align*}

Finally, we can conclude the construction by defining
the support function $\psi \in \Lip(\Tn)$,
\begin{align*}
    \psi(x) = 
    \begin{cases}
    -\zeta(d_{\Omega_+}) & \text{in $\Omega_+$},\\
    0 & \text{in $D := \cl \Omega_- \setminus \Omega_+$},\\
    -\zeta(d_{\Omega_-}) & \text{in $\cl \Omega_-^c$},
    \end{cases}
\end{align*}
and the vector field
\begin{align*}
z(x) = -\nabla \bra{\ta(d_{\Omega_-}(x)) + \ta(d_{\Omega_+}(x))}
     = -\ta'(d_{\Omega_-}(x)) \nabla d_{\Omega_-}(x)
       - \ta'(d_{\Omega_+}(x)) \nabla d_{\Omega_+}(x).
\end{align*}
By construction, $z \in C^1(\Tn)$ and therefore $z \in X_2(\Tn)$.
Indeed,
note that
$\ta'(d_{\Omega_\pm}(x)) \neq 0$ only in the $\de$-neighborhood
$K_\pm = \set{\abs{d_{\Omega_\pm}} < \de}$,
and $d_{\Omega_\pm} \in C^2(K_\pm)$.
Moreover, as
$K_+ \cap K_- = \emptyset$, 
$\ta' \leq 1$
and $\abs{\nabla d_{\Omega_\pm}} \leq 1$ a.e.,
we have $\abs{z} \leq 1$ in $\Tn$.
This also shows that $z$ satisfies the boundary condition
\eqref{2.10} and thus $z$ is a Cahn-Hoffman field in $D$.
In particular,
the variational problem \eqref{vp-facet} has a unique solution
$w_0 \in L^2(D)$.

Finally, a simple calculation yields that $z = \nabla \psi$
and $\abs{z} = \abs{\nabla \psi} = 1$
whenever $\nabla \psi \neq 0$.
Hence \eqref{lip-subdiff} holds
and therefore $\psi \in \mathcal{D} (\partial E)$
and $\divo z \in -\partial E(\psi)$.

Let $w_0$ be the solution of the variational problem \eqref{vp-facet}.
There exists a Cahn-Hoffman vector field $z_0$ in $D$
such that $w_0 = \divo z_0$ in $D$.
Let $\psi$ be an arbitrary support function of $(\Omega_-, \Omega_+)$
such that $\psi \in \mathcal{D}(\partial E)$.
Again, there exists $z_\psi \in X_2(\Tn)$
which satisfies \eqref{lip-subdiff}
with $w = - \partial^0 E(\psi)$.
Clearly $z_\psi \in X_2(\interior D)$
and $z_\psi$ satisfies the boundary condition \eqref{2.10}
due to Proposition~\ref{pr:trace}
and it is therefore a Cahn-Hoffman vector field in $D$.
Moreover, the vector field
\[
\bar z(x) = \begin{cases}
z_0(x) & \text{on } D,\\
z_\psi(x) & \text{on } D^c
\end{cases}
\]
belongs to $X_2(\Tn)$,
and satisfies the inclusion in \eqref{lip-subdiff}.
Therefore $\norm{\divo z_\psi}_{L^2(D)} \leq \norm{\divo z_0}_{L^2(D)}$
since $\divo z_\psi = -\partial^0 E(\psi)$.
But the uniqueness of the variational problem yields
\[
-\partial^0 E(\psi) = \divo z_\psi = w_0 \qquad\text{a.e. on } D.
\]
This holds for an arbitrary support function
$\psi \in \mathcal{D}(\partial E)$.
\qedhere\end{proof}

\begin{remark}
\label{re:support-scal-invar}
It is easy to see that
if $\psi \in \mathcal{D}(\partial E)$ is a support function
of a smooth pair $(\Omega_-, \Omega_+)$
then,
for any strictly increasing
$\ta \in C^1((-\infty, 0]) \cap C^1([0, \infty))$,
$\ta \circ \psi \in \mathcal{D}(\partial E)$
with the same Cahn-Hoffman vector field $z_0$ as $\psi$
on $\Omega_- \setminus \cl \Omega_+$,
and $\ta \circ \psi$ is a support function of $(\Omega_-, \Omega_+)$.
This fact will be used multiple times
with the particular choice $\ta(s) = \al s_+ - \be s_-$
for some $\al, \be > 0$.
\end{remark}

We are now ready to define the nonlocal curvature
for a smooth pair $(\Omega_-,\Omega_+)$.

\begin{definition}
Let $(\Omega_-,\Omega_+)$ be a smooth pair.
Let $w_0$ be the unique minimizer of \eqref{vp-facet}.
Then we call $w_0$ the \textit{nonlocal curvature} of $(\Omega_-,\Omega_+)$,
and we denote it by $\Lambda(\Omega_-, \Omega_+)$.
\end{definition}
Although there is a huge literature on obstacle problems
(see e.g. \cite{Rod}),
for our obstacle problem \eqref{vp-facet}
there seems to be no literature when
$\Omega_+$ is non-empty except in the one-dimensional case
in which $w_0$ is always constant.
(If one end of $\Omega_-$ is not covered by $\cl{\Omega}_+$,
$w_0$ must be zero near that point.)
In general, it is hard to calculate $w_0$.
If there is no obstacle condition $|z|\leq 1$ in (\ref{2.9}),
the minimizer $w_0$ is always constant
since its Euler-Lagrange equation is $\nabla\textup{div}\,z_{0}=0$.
If the minimizer $w_0$ is a constant $\lambda$,
the value of $\lambda$
is easily computed by the Green's theorem,
Theorem~\ref{th:green}.
Indeed,
\begin{align*}
\mathcal{H}^{n}(D)\lambda
    &= \int_D \divo z_{0} \dx
    = \int_{\partial\Omega_-} z_0\cdot\nu_{\partial\Omega_-}
     \;d\mathcal{H}^{n-1}
    + \int_{\partial\Omega_+} z_0 \cdot(-\nu_{\partial\Omega_+})
     \;d\mathcal{H}^{n-1}\\
    &= -\mathcal{H}^{n-1}(\partial\Omega_{-})
        +\mathcal{H}^{n-1}(\partial\Omega_{+})
\end{align*}
and therefore
$$\lambda=\frac{1}{\mathcal{H}^{n}(D)}\left(\mathcal{H}^{n-1}(\partial\Omega_+)-\mathcal{H}^{n-1}(\partial\Omega_-)\right),$$
where $\mathcal{H}^n$ denotes the $n$-dimensional Hausdorff measure
so that $\mathcal{H}^n$ is the Lebesgue measure. 

The question is whether or not $w_0$ is constant on $D$.
When $\Omega_+$ is empty, this problem has been well-studied
by Bellettini et al. \cite{BNP01IFB}.
A facet where $w_0$ is constant is called \emph{calibrable}.
There is a necessary and sufficient condition for calibrability.
A sufficient condition for a convex set to be calibrable is that
the principal curvature of $\partial\Omega_-$ is bounded by
$\mathcal{H}^{n-1}(\partial\Omega_-)/\mathcal{H}^{n}(D)$.

In the case when $w_0$ is not constant, not much is known except regularity \cite{BNP01a,BNP01b}.
Among other results, they proved that the minimizer is in
$BV(\Omega_-)$ but may be discontinuous.
See also \cite[Theorem~5.16]{B10}.

Let us point out here that our $(\Omega_-,\Omega_+)$ pair formulation does not require that $D$ is connected so it is more general.

We conclude this subsection by illustrating
the definitions on a few examples
of simple smooth pairs.

\begin{example}
Let $\Omega_- = \Tn$ and $\Omega_+ = \emptyset$,
i.e., the facet $D = \cl \Omega_- \setminus \Omega_+$
is the whole torus. In this case $\Lambda(\Omega_-, \Omega_+) \equiv 0$
on $D = \Tn$ since $z \equiv 0$ 
is a Cahn-Hoffman vector field in $D$.
\end{example}

\begin{example}
\label{ex:ball-solution}
Let $\Omega_- = B_{R}(0)$, $R \in (0, 1/2)$, and let $\Omega_+ = \emptyset$. Then 
\begin{align*}
z(x) = \frac{-x}{R}
\end{align*}
is a Cahn-Hoffman vector field on $D$. Note that the minimizer is $w_0 \equiv \frac{-n}{R}$.
\end{example}

\begin{example}
Let $\Omega_{\pm}=B_{R_\pm}(0)$ with $0<R_+<R_- < 1/2$.
The minimizer $w_0$ of \eqref{vp-facet} for this pair $(\Omega_-,\Omega_+)$
is the constant
\[\Lambda(\Omega_-,\Omega_+)=\frac{1}{\mathcal{H}^n(D)}\left(\mathcal{H}^{n-1}(\partial\Omega_+)-\mathcal{H}^{n-1}(\partial\Omega_-)\right).\]
In particular,
the facet $D$ is calibrable even though $\Omega_+\neq\emptyset$.
\end{example}

\begin{proof}
We shall construct a Cahn-Hoffman vector field $z$ so that
$\divo z$ is a constant.
In fact, we take $z$ of the form $z = \nabla \vp$, where $\vp$ is a radially symmetric solution of the Poisson equation with Neumann boundary data. Let us denote
\begin{align*}
b := \Lambda(\Omega_-, \Omega_+) = n \frac{R_-^{n-1} - R_+^{n-1}}{R_-^n - R_+^n}.
\end{align*}
$\vp$ is a solution of
\begin{align*}
\begin{cases}
\Delta \vp = b & \text{in } D,\\
\pd{\vp}{\nu} = \pm 1& \text{on } \abs{x} = R_\pm. 
\end{cases}
\end{align*}
In fact, $\vp$ might be chosen in the form
\begin{align*}
\vp(x) = \frac{a}{2-n} \abs{x}^{2-n} + \frac{b}{2} \abs{x}^2.
\end{align*}
To satisfy the boundary condition, we take
\begin{align*}
a = (R_+R_-)^{n-1} \frac{R_- - R_+}{R_-^n - R_+^n} > 0.
\end{align*}
A simple computation yields that
\begin{align*}
z = \ta(\abs{x}) \frac{x}{\abs{x}},
\end{align*}
where $\ta(\abs{x}) = a \abs{x}^{1-n} + b \abs{x} > 0$. Since $\ta(\rho)$ is convex for $\rho > 0$ and $\ta(R_\pm) = 1$, we conclude that $\abs{z} \in (0, 1]$ in $D$. Therefore $z$ is a Cahn-Hoffman vector field.
\qedhere\end{proof}

\subsection{Monotonicity of nonlocal curvatures}

In this subsection we shall prove the monotonicity
of nonlocal curvature $\Lambda$
with respect to a smooth pair $(\Omega_-,\Omega_+)$.
Its precise form is given by the following theorem.

\begin{theorem}[Monotonicity]
\label{th:monotonicity}
Let $(\Omega^i_-, \Omega^i_+)$ be a smooth pair for $i=1,2$.
Let $\Lambda^i=\Lambda(\Omega^i_-,\Omega^i_+)$
be the nonlocal curvature
on $D^i=\cl{\Omega^i_-} \setminus \Omega^i_+$ for $i=1,2$.
If $\Omega^1_-\subset\Omega^2_-$ and $\Omega^1_+\subset\Omega^2_+$,
then $\Lambda^1\leq\Lambda^2$ a.e. in $D^1\cap D^2$.
\end{theorem}

\begin{proof}
We apply the comparison principle for the resolvent problem
\eqref{resolvent-problem};
a similar idea was pursued in \cite{GGM},
where the evolution equation \eqref{2.4}
is used instead of the resolvent problem.

We have
\[\Lambda^i = -\partial^0 E (\psi^i)\quad\textup{a.e. in}\quad D^i,\]
where $\psi^i \in \dom(\partial E)$
are the support functions for $(\Omega^i_-, \Omega^i_+)$ constructed in
Proposition~\ref{pr:support-function}.
Observe that it is possible to take the same $\de > 0$
in the proof of Proposition~\ref{pr:support-function}
for both $\psi^1$ and $\psi^2$,
and then by construction
and due to the ordering $\Omega^1_\pm \subset \Omega^2_\pm$
we have
\begin{align}
\label{psi-order}
\psi^1 &\leq \psi^2 &&\text{in $\Tn$,}
&\psi^1 &= \psi^2 = 0 &&\text{in $D^1 \cap D^2$}.
\end{align}

For each $a > 0$,
we consider the solution $f^i_a$ of
the resolvent problem \eqref{resolvent-problem}
with right-hand side $\psi^i$.
Due to Proposition~\ref{pr:resolvent-convergence},
$(f^i_a - \psi^i)/a \to -\partial^0 E(\psi^i)$ in $L^2$ as $a \to 0$.
Therefore there is a subsequence $a_k \to 0$ as $k \to \infty$
such that
$(f^i_{a_k} - \psi^i)/a_k \to -\partial^0 E(\psi^i)$ a.e.
as $k \to \infty$ for $i = 1, 2$.

The comparison principle, Theorem~\ref{th:resolvent-comparison},
and \eqref{psi-order}
imply that $f^1_{a_k} \leq f^2_{a_k}$
and
$f^i_{a_k} - \psi^i = f^i_{a_k}$ in $D^1\cap D^2$
for all $k$.
Therefore
\begin{align*}
\Lambda^1 &= - \partial^0 E(\psi^1)
    = \lim_{k\to\infty} \frac{f^1_{a_k}}{a_k}
    \leq \lim_{k\to\infty} \frac{f^2_{a_k}}{a_k}
    = - \partial^0 E(\psi^2)
    = \Lambda^2
    &&\text{a.e. in $D^1 \cap D^2$}
\end{align*}
and the monotonicity of $\Lambda$ is established.
\qedhere\end{proof}

To illustrate the previous theorem,
we show some basic bounds on
$\Lambda(\Omega_-, \Omega_+)$.

\begin{proposition}
Let $(\Omega_-, \Omega_+)$ be a smooth pair.
Then 
\begin{align*}
\abs{\Lambda(\Omega_-, \Omega_+)(x)}
    \leq \frac{n}{\rho(x)} \quad \text{for a.e. } x \in D := 
        \cl\Omega_- \setminus \Omega_+,
\end{align*} 
where 
$\rho(x) :=
    \sup \set{\rho < 1/4 : x \in B_\rho(y) \subset D \text{ for some } y}$.
\end{proposition}

\begin{proof}
Let $\psi$ be the support function of $(\Omega_-, \Omega_+)$ from
Proposition~\ref{pr:support-function}
and set $M = \max \psi$.

We fix $\xi \in D$
and choose $\rho \in (0,1/4)$ and $y$ for which $\xi \in B_\rho(y) \subset \cl{B}_\rho(y) \subset \interior D$.
There also exists $\de \in (0,1/4)$
such that $\cl{B}_{\rho + \de}(y) \subset D$. 

We introduce the barrier $\vp \in \Lip(\Tn)$, 
\begin{align*}
\vp(x) = 
\begin{cases}
0 & \abs{x - \xi} \leq \rho,\\
M \frac{\abs{x- \xi} - \rho}{\de} & \rho < \abs{x - \xi} < \rho + \de,\\
M & \abs{x - \xi} \geq \rho+\de.
\end{cases}
\end{align*}
This function is a support function of the smooth pair
$(\Tn, A)$,
where $A = \Tn \setminus \cl B_\rho(\xi)$.
As in Example~\ref{ex:ball-solution}, it can be shown that $\vp \in \mathcal{D} (\partial E)$;
that example also shows that
$\Lambda(\Tn, A) = \frac{n}{\rho}$ on $B_\rho(y)$.
Clearly, $\Omega_+ \subset A$ and therefore
the monotonicity result in Theorem~\ref{th:monotonicity}
allows us to compare
$\Lambda(\Omega_-, \Omega_+) \leq \Lambda(\Tn, A) = \frac{n}{\rho}$
a.e. in $B_\rho(\xi)$.
The lower bound is analogous.
\qedhere\end{proof}

\section{Viscosity solutions}
\label{sec:viscosity-solutions}

In this section we finally introduce a new
notion of viscosity of \eqref{tvf}.
As in the previous works,
it is necessary to carefully chose an appropriate class
of test functions.
To this end we introduce admissible faceted test functions.

\begin{definition}
Function $\vp(x,t) = f(x) + g(t)$,
where $f \in \Lip(\Tn)$ and $g \in C^1([\hat t - \de, \hat t + \de])$
for some $\de > 0$ and $\hat t \in (0,T)$,
is called \emph{admissible faceted test function
at $(\hat x, \hat t) \in Q := \Tn \times (0,T)$
with pair $(\Omega_-, \Omega_+)$}
if
$f$ is a support function of a smooth pair $(\Omega_-, \Omega_+)$
and $\hat x \in \Omega_- \setminus \cl\Omega_+$.
The pair $(\Omega_-, \Omega_+)$ is called the \emph{pair associated
with $\vp$ at $(\hat x, \hat t)$}.
\end{definition}

\begin{definition}
\label{def:general-position}
We say that an admissible faceted test function $\vp$
at $(\hat x, \hat t) \in Q$ with pair $(\Omega_-, \Omega_+)$
is \emph{in general position of radius $\eta > 0$
with respect to $u : \cl Q \to \R$ at $(\hat x, \hat t)$}
if, for all $h \in \cl B_\eta(0)$,
we have
$\hat x + h \in \Omega_- \setminus \cl\Omega_+$
and
\begin{align*}
u(x - h, t) - \vp(x,t) &\leq u(\hat x, \hat t) - \vp(\hat x, \hat t)
& &\text{for all $x \in \Tn$, $t \in [\hat t - \eta, \hat t +\eta]$.}
\end{align*}
If such $\eta > 0$ exists, we simply say that $\vp$ is
in general position with respect to $u$ at $(\hat x, \hat t)$.
\end{definition}

The notion of general position 
formalizes for admissible test functions
the idea that two facets stay ordered
in the sense of Theorem~\ref{th:monotonicity}
even when one of them is shifted
by a small distance in an arbitrary direction.
This makes the family of test functions smaller
and turns out to be the right ingredient to
obtain the existence of solutions through stability,
without affecting the comparison principle.

\begin{definition}[Viscosity solutions]
\label{def:Lambda-subsol}
An upper semi-continuous function $u : \cl Q \to \R$
is a \emph{viscosity subsolution}
of \eqref{tvf}
if
the following holds:

\begin{enumerate}[(i)]

\item (\emph{faceted test})
If $\vp$ is an admissible faceted test function such that $\vp$ is
in general position of radius $\eta$
with respect to $u$
at a point $(\hat x, \hat t) \in Q$
then there exists $\de \in (0, \eta)$ such that
\begin{align*}
\vp_t(\hat x, \hat t)
    + F\pth{0, \essinf_{B_\de(\hat x)}\Lambda(\Omega_-, \Omega_+)} \leq 0,
\end{align*}
where $(\Omega_-, \Omega_+)$ is the smooth pair
associated with $\vp$ at $(\hat x, \hat t)$.

\item (\emph{conventional test})
If $\vp \in C^{2,1}_{x,t}(U)$ in a neighborhood $U \subset Q$ of
a point
$(\hat x, \hat t)$,
such that $u - \vp$ has a local maximum at $(\hat x, \hat t)$
and $\abs{\nabla \vp} (\hat x, \hat t) \neq 0$, then
\begin{align*}
\vp_t(\hat x, \hat t)
    + F\pth{\nabla \vp(\hat x, \hat t),
        k(\nabla \vp(\hat x,\hat t),
        \nabla^2 \vp(\hat x,\hat t))} \leq 0,
\end{align*}
where $\nabla^2$ is the Hessian and
\begin{align}
\label{Lambda-nondeg}
    k(p, X) &= \frac1{\abs{p}}
        \trace {X \pth{I - \frac{p \otimes p}{\abs{p}^2}}}
        && \text{for $p \in \Rn \setminus \set0$,
            $X \in \mathcal{S}^n$},
\end{align}
so that 
$k(\nabla \vp(\hat x, \hat t), \nabla^2 \vp(\hat x, \hat t))
    = \bra{\divo (\nabla\vp / \abs{\nabla\vp})}(\hat x, \hat t)$.
Here $\mathcal{S}^n$ is the set of $n \times n$-symmetric
matrices.
\end{enumerate}

\emph{Viscosity supersolution} can be defined similarly as
a lower semi-continuous function, replacing maximum by minimum,
$\leq$ by $\geq$, and $\essinf$ by $\esssup$. Furthermore, in (i)
$\vp$ must be such that $-\vp$ is in general position of radius $\eta$
with respect to $-u$
(see also Remark~\ref{re:support-function-symmetry}).

Function $u$ is a \emph{viscosity solution}
if it is both a subsolution and supersolution.
\end{definition}

\begin{remark}
Our definition is essentially similar to that of local
version in \cite{GG98ARMA}
where one-dimensional problem is studied for more general equation
$u_t + F(u_x, (W'(u_x))_x) = 0$
with convex $W$ whose derivative has a discrete set of jumps,
not necessarily a singleton.
\end{remark}

\section{Comparison principle}
\label{sec:comparison-principle}

Our next task is to prove the comparison principle
on the space-time cylinder $Q := \T^n \times (0,T)$
for any fixed $T > 0$.

\begin{theorem}[Comparison]
\label{th:comparison}
Let $u$ and $v$
be respectively a bounded viscosity subsolution
and a bounded viscosity supersolution
of \eqref{tvf} on $Q$.
If $u \leq v$ at $t = 0$ then $u \leq v$ on $Q$.
\end{theorem}

We will start with a variation of the standard
doubling-of-variables argument.
As in \cite{GG98ARMA}, we define 
\begin{align*}
w(x,t,y,s) = u(x,t) - v(y,s)
\end{align*} 
and, given positive constants $\e, \si, \ga$
and a vector $\zeta \in \Tn$,
the functions
\begin{align*}
\Psi_\zeta(x,t,y,s; \e,\si,\ga) &:= 
    \frac{\abs{x - y -\zeta}^2}{2\e} + S(t,s; \si, \ga),\\
S(t,s; \si, \ga) &:= 
    \frac{\abs{t -s}^2}{2\si} + \frac{\gamma}{T - t}
    + \frac{\gamma}{T - s},
\end{align*}
where $\abs{x - y - \zeta}$ is interpreted
as in \eqref{dist-on-torus}.

We analyze the maxima of functions
\begin{align*}
\Phi_\zeta(x,t,y,s;\e,\si,\ga)
    := w(x,t,y,s) - \Psi_\zeta(x,t,y,s;\e,\si,\ga).
\end{align*}
The purpose of introducing an extra parameter $\zeta$
is to restrict the class of possible contact points
that require a construction of an admissible faceted test function
in general position (Definition~\ref{def:Lambda-subsol}(i))
down to the profiles that allow for small shifts.
Indeed, if at least one maximum of $\Psi_\zeta$ for small $\zeta$ occurs
at a point $(x, t, y, s)$ such that
$x - y \neq \zeta$ then $\grad_x \Phi_\zeta \neq 0$ at the point
of maximum,
and we can apply the classical viscosity argument and
proceed with the well-known construction of
a test function for Definition~\ref{def:Lambda-subsol}(ii).
On the other hand, if all maxima of $\Phi_\zeta$ for small $\zeta$ 
occur at points with $x - y = \zeta$,
it is necessary to proceed with a more involved construction of
an admissible faceted test function in general position.
Fortunately, the structure of maxima with respect to $\zeta$
provides enough information to infer ``flatness'' 
of $u$ and $v$ at the contact point,
which in turn guarantees that the construction is possible.
We carefully analyze the profiles of $u$ and $v$
at the level of the contact,
construct two smooth pairs that represent facets
ordered in the sense of Theorem~\ref{th:monotonicity},
and construct admissible faceted test function
with the help of the support function introduced
in Proposition~\ref{pr:support-function}.
The definition of viscosity solution then
yields a contradiction.

\bigskip
Therefore, suppose that the comparison principle does not hold,
i.e.,
\begin{align*}
m_0 := \sup_{(x,t)\in Q} \bra{u(x,t) - v(x,t)} > 0.
\end{align*}
We have the following proposition:
\begin{proposition}
\label{pr:small-constants}
Set $M := \max_{\cl Q \times \cl Q} w \geq m_0$,
$m_0' := \frac78 m_0$,
$\kappa(\e) :=
    \frac12\pth{\e \pth{m_0 - m_0'}}^{1/2}$.

There exist positive constants $\e_0, \si_0$ and $\ga_0$
such that
for any
$\e \in (0, \e_0)$, $\si \in (0, \si_0)$, $\ga \in (0,\ga_0)$
and $\zeta\in \Tn$ such that $\abs{\zeta} \leq \kappa(\e)$
the following is true:

If $(x_\zeta, t_\zeta, y_\zeta, s_\zeta)
\in \argmax_{\cl Q \times \cl Q} \Phi_\zeta(\cdot; \e, \si, \ga)$ then
\begin{enumerate}
\romanlist
\item $(x_\zeta, t_\zeta, y_\zeta, s_\zeta) \in Q \times Q$,
\item $\abs{t_\zeta - s_\zeta} \leq (M \si)^{1/2}$,
      $\abs{x_\zeta - y_\zeta - \zeta} \leq (M\e)^{1/2}$,
\item $\sup \Phi_\zeta > m_0'$.
\end{enumerate}
\end{proposition}

\begin{proof}
See \cite[Proposition~7.1, Remark~7.2]{GG98ARMA}.
In our case, we take $\gamma_0 := \min(\gamma_0, \gamma_0')$ for simplicity.
\qedhere\end{proof}

The possible scenarios can be divided into two cases
that will be treated independently:
\begin{itemize}
\item Case I: 
all the points of maximum lie on the space diagonals, i.e.,
\begin{align}
\label{case-I-argmax}
\argmax_{\cl Q \times \cl Q} \Phi_\zeta
    \subset \set{(x,t,y,s) \in \cl Q \times \cl Q : x - y = \zeta}
\end{align}
for all $\e, \si, \ga$ and $\abs{\zeta} < \kappa(\e)$.

\item Case II: 
there exists $\e, \si, \ga$, $\abs{\zeta} < \kappa(\e)$
and $(x,t,y,s) \in \argmax_{\cl Q \times \cl Q} \Phi_\zeta$
such that $x - y \neq \zeta$.
\end{itemize}

\subsection{Case I}

Let us fix
$\e \in (0, \e_0)$, $\si \in (0, \si_0)$, $\ga \in (0, \ga_0)$
such that \eqref{case-I-argmax} holds for all $\zeta$, $\abs\zeta \leq
\kappa (\e)$.
To simplify the notation, we set $\la := \kappa(\e)/2$.
We also do not explicitly state the dependence of
the following formulas on the fixed $\e$, $\si$ and $\ga$.
Let us define
\begin{align*}
\ell(\zeta) = \sup_{\cl Q \times \cl Q} \Phi_\zeta.
\end{align*}

The first step is the application of
the constancy lemma \cite[Lemma 7.5]{GG98ARMA},
as in the proof of \cite[Proposition 7.6]{GG98ARMA},
to show that $\ell(\zeta)$ is constant for $\abs{\zeta} \leq \la$.

\begin{lemma}[Constancy lemma]
\label{le:constancy}
Let $K$ be a compact set in $\R^N$ for some $N > 1$
and let $h$ be a real-valued upper semi-continuous function on $K$.
Let $\phi$ be a $C^2$ function on $\R^d$ with $1 \leq d < N$.
Let $G$ be a bounded domain in $\R^d$.
For each $\zeta \in G$ assume that
there is a maximizer
$(r_\zeta, \rho_\zeta)\in K$ of
\begin{align*}
    H_\zeta(r, \rho) = h(r, \rho) - \phi(r - \zeta)
\end{align*}
over $K$ such that $\nabla \phi(r_\zeta - \zeta) = 0$.
Then,
\begin{align*}
    h_\phi(\zeta) = \sup\set{H_\zeta(r,\rho): (r, \rho)\in K}
\end{align*}
is constant on $G$.
\end{lemma}

We apply Lemma~\ref{le:constancy}
with the following parameters:
\begin{gather*}
N = 2n + 2, \quad
d= n, \quad
\rho = (y, t, s) \in \Tn \times \R \times \R,\\
K = 
    \set{(x - y, y, t, s):
       (x,y) \in \Tn \times \Tn,
       (t,s) \in [0,T] \times [0,T]},\\
G = B_{2\la}(0),\\
h(r,\rho) = w(r + y, t, y, s) - S(t,s),
\quad
\phi(r) = \frac{\abs{r}^2}{2\e}.
\end{gather*}
$K$ can be treated as a compact subset of $\Rn$
in a straightforward way.
We infer that $\ell(\zeta) = h_\phi(\zeta)$ is constant
for $\abs{\zeta} \leq \la$.

Therefore we have also an ordering
analogous to \cite[Corollary 7.9]{GG98ARMA}:

\begin{lemma}
\label{le:orderInDoubling}
Let $(\hat x, \hat t, \hat x, \hat s) \in \argmax \Phi_0$. Then
\begin{align*}
u(x,t) - v(y,s) - S(t,s)
    \leq u(\hat x, \hat t) - v(\hat x, \hat s) - S(\hat t, \hat s)
\end{align*}
for all $s,t \in (0,T)$
and $x,y \in \Tn$ such that $\abs{x -y} \leq \la := \kappa(\e)/2$.
\end{lemma}

\begin{proof}
For $\abs{x-y} \leq \la$,
a straightforward computation using the constancy of $\ell$
yields
\begin{align*}
u(x,t) - v(y,s) - S(t,s)
    &= \Phi_{x - y}(x,t,y,s) \\
    &\leq \ell(x-y) 
    =\ell(0) \\
    &= u(\hat x, \hat t) - v(\hat x, \hat s) - S(\hat t, \hat s).
    \quad \qedhere
\end{align*}
\end{proof}

We will use Lemma~\ref{le:orderInDoubling} to create room
for ordered smooth facets,
whose support functions will play the role of admissible test functions
for $u$ and $v$.
The test functions themselves do not have to be ordered
due to the positive 0-homogeneity of the nonlocal curvature in our setting
(recall Remark~\ref{re:support-scal-invar}).

At this moment,
we will fix
$(\hat x, \hat t, \hat x, \hat s) \in
\argmax_{\cl Q \times \cl Q} \Phi_0$
and we set
\begin{align*}
\al &:= u(\hat x, \hat t), & \be &:= v(\hat x, \hat s).
\end{align*}
Let us define the closed sets 
\begin{align*}
U &:= \set{u(\cdot, \hat t) \geq \al},
    & V &:= \set{v(\cdot, \hat s) \leq \be},
\end{align*}
see Figure~\ref{fig:comparison-facets}.
\begin{figure}
\centering
\fig{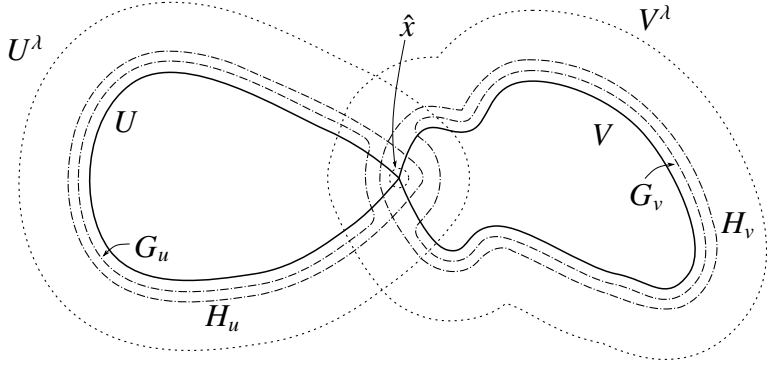}{4.5in}
\caption{Facet construction for $u$ and $v$ in the proof of
the comparison theorem.}
\label{fig:comparison-facets}
\end{figure}

It is convenient to introduce the notion of dilation and erosion
of a set in morphology.
For any set $A \subset \Tn$ and $\rho \in \R$, $A^\rho$ denotes
the generalized neighborhood
\begin{align}
\label{set-neighborhood}
 A^\rho := 
 \begin{cases}
 A + \cl B_\rho(0) & \rho > 0,\\
 A & \rho = 0,\\
 \set{x \in A : \cl B_{\abs\rho}(x) \subset A} & \rho < 0.
 \end{cases}
 \end{align} 
Here $A+B = \set{x+y: x \in A,\ y \in B}$ is the vector sum
called the Minkowski sum.
In image analysis it is often written as
$A^\rho = A \oplus \cl B_\rho(0)$ for $\rho > 0$
and $A^{\rho} = A \ominus \cl B_{\abs\rho}(0)$ for $\rho < 0$,
where $\oplus$ denotes the Minkowski addition and $\ominus$
denotes the Minkowski decomposition.
In morphology, $\oplus$ is called dilation
and $\ominus$ is called erosion.
If $A$ is open and $\rho > 0$, we have
\begin{align*}
A^\rho &= \set{x \in \Tn : \dist(x, A) < \de},\\
A^{-\rho} &= \set{x \in A : \dist(x,\partial A) > \de}.
\end{align*}
Clearly, for any $\rho \in \R$,
the generalized neighborhood $A^\rho$ is open (resp. closed)
provided that
$A$ itself is open (resp. closed).
Moreover
\begin{align*}
A^{-\rho} = \pth{\pth{A^c}^\rho}^c,
\end{align*}
where $A^c$ denotes the complement of $A$, i.e., $A^c := \T^n \setminus A$.
If $A$ has the interior ball property with radius $r > 0$
then $\pth{A^{-\rho}}^\rho = A$ for any $0 < \rho \leq r$. 

\begin{corollary}
\label{co:extended-ordering}
We have
\begin{align*}
u(x, t)
    \leq \al + S(t, \hat s) - S(\hat t, \hat s)
    \qquad \text{for }  x\in V^\la,\ t \in (0,T),
\end{align*}
and 
\begin{align*}
v(y, s)
    \geq \be + S(\hat t, s) - S(\hat t, \hat s)
    \qquad \text{for }  y\in U^\la,\ s \in (0,T).
\end{align*}
\end{corollary}

\begin{proof}
If $x\in V^\la$ then there exists $y \in V$ with $\abs{x - y} \leq \la$ and the conclusion follows from Lemma~\ref{le:orderInDoubling}
and the definition of $V$.
The proof of the second inequality is analogous.
\qedhere\end{proof}

In what follows,
we will need to construct smooth sets
that approximate a given general set.
The next lemma provides such approximation.

\begin{lemma}
\label{le:smooth-nbd}
For any set $E \subset \Tn$
and constants $r > 0$, $\de > 0$,
there exists an open set $G \subset \Tn$
with a smooth boundary
such that
\begin{align*}
\cl E^r \subset G \subset \cl G \subset \interior (E^{r+\de}).
\end{align*}
\end{lemma}

\begin{proof}
This is a simple consequence of the Sard's theorem.
First,
if $E^{r + \de} = \Tn$ (resp. $E = \emptyset$)
then we choose $G = \Tn$ (resp. $G = \emptyset$)
and we are done.
Thus we may assume that
$\emptyset \neq E \subset E^{r + \de} \subsetneq \Tn$.
Let $d(x) := \dist(x, E)$ and let $\phi_{\de/4}$
be the standard smooth mollifier with support in $B_{\de/4}$.
Let $g = d * \phi_{\de/4}$.
We observe that $g \in C^\infty$,
and, by triangle inequality, also
$g \leq r + \frac14 \de$ on $\cl E^r$ and $g \geq r + \frac34 \de$ on
$(\interior(E^{r + \de}))^c$.

Sard's theorem yields that $\abs{g(\set{x : \nabla g(x) = 0})} = 0$,
where $\abs{\cdot}$ stands for the one-dimensional
Lebesgue measure.
Therefore, there exists $\rho \in (r + \frac14 \de, r + \frac34 \de)$
such that $\nabla g \neq 0$ on $\set{g = \rho}$.
We set $G = \set{x: g(x) < \rho}$
and observe that $\partial G = \set{g = \rho} \in C^\infty$
by the implicit function theorem.
\qedhere\end{proof}

We continue with the construction of smooth pairs for $u$ and $v$,
both containing $\hat x$.
We set $r := \frac1{20} \la$ and consider the closed sets
\begin{align*}
Z &:= \cl{U \setminus V^{15r}}, & W &:= \cl{V \setminus U^{15r}}.
\end{align*}

By Lemma~\ref{le:smooth-nbd},
there are open sets
$G_i$, $H_i$, $i = u, v$,
with smooth boundary
such that
\begin{subequations}
\begin{align}
\label{AuBu}
Z^{3r} &\subset  G_u \subset \subset \interior (Z^{4r})
    \subset Z^{4r}
    \subset U^{5r} \subset H_u \subset \subset \interior (U^{6r}),\\
W^{3r} &\subset  G_v \subset \subset \interior (W^{4r})
    \subset W^{4r}
    \subset V^{5r} \subset H_v \subset \subset \interior (V^{6r}).
\end{align}
\end{subequations}
If $Z$ (resp. $W$) is empty then we take
$G_u = \emptyset$ (resp. $H_u =\emptyset$).

Because, by construction, $Z^{4r} \cap V^{6r} = \emptyset$ as well as
$W^{4r} \cap U^{6r} = \emptyset$, we also have
\begin{align}
\label{dist-of-facets}
\dist(G_u, H_v) > 0, \qquad \dist(H_u, G_v) > 0.
\end{align}

Our next task is to construct
admissible faceted test functions around $(\hat x, \hat t)$
and $(\hat x, \hat s)$
for $u$ and $v$, respectively,
with facets given by the smooth pairs
\begin{align}
\label{def-of-facets}
(\Omega_{-,u}, \Omega_{+,u}) := (H_u, G_u) \qquad
\text{and} \qquad
 (\Omega_{-,v}, \Omega_{+,v}) :=  (\interior G_v^c, \interior H_v^c).
\end{align}
Note that these smooth pairs satisfy the assumptions
of Theorem~\ref{th:monotonicity},
in particular, the facets are ordered due to \eqref{dist-of-facets}.

Consider the closed sets 
\begin{align*}
X &= \set{x : \dist(x, U) \geq r}, & Y &= \set{x : \dist(x, V) \geq r}.
\end{align*}
Since $u(\cdot, \hat t) < \al$ in $X$ and $v(\cdot, \hat s) > \be$ in $Y$,
and $u$ and $v$ are semicontinuous,
there exists $\tau > 0$ such that
\begin{subequations}
\begin{align}
\label{u-upper-strict-bound}
u(\cdot, t) &< \al + S(t, \hat s) - S(\hat t, \hat s)
    \quad \text{in } X
        \text{ for } t \in [\hat t- \tau, \hat t + \tau],\\
v(\cdot, s) &> \be + S(\hat t, \hat s) - S(\hat t, s)
    \quad \text{in } Y 
    \text{ for } s \in [\hat s- \tau, \hat s + \tau].
\end{align}
\end{subequations}
We note that if both $X$ and $Y$ are empty we choose some suitable small $\tau > 0$.
Therefore, together with Corollary~\ref{co:extended-ordering},
\begin{subequations}
\begin{align}
\label{lambda-extened-ordering}
u(\cdot, t) &\leq \al + S(t, \hat s) - S(\hat t, \hat s) \quad \text{in } X \cup V^\la \text{ for } t \in [\hat t- \tau, \hat t + \tau],\\
v(\cdot, s) &\geq \be + S(\hat t, \hat s) - S(\hat t, s) \quad \text{in } Y \cup U^\la \text { for } s \in [\hat s- \tau, \hat s + \tau].
\end{align}
\end{subequations}

\begin{lemma} The choice of the sets $G_i$ and $H_i$ guarantees that
\begin{subequations}
\begin{align}
\label{AuBuXY}
(G_u)^{-2r} \cup X \cup V^\la = (H_u)^{-r} \cup X = \Tn,\\
(G_v)^{-2r} \cup Y \cup U^\la = (H_v)^{-r} \cup Y = \Tn.
\end{align}
\end{subequations}
\end{lemma}

\begin{proof}
We will only show \eqref{AuBuXY}, the other equality
is analogous.
Suppose that there exists $x \notin (G_u)^{-2r} \cup X$.
We will show $x \in V^\la$.
From the definition of $X$ and from \eqref{AuBu}, it follows that
$Z^r \subset (G_u)^{-2r}$ and thus
\begin{align}
\label{dist-x-Z}
\dist(x, U) &< r, & \dist(x, Z) &= \dist(x, U \setminus V^{15r}) > r.
\end{align}
Therefore there exists $y \in U$ with $\abs{x - y} < r$,
and the second inequality in \eqref{dist-x-Z} implies that $y \in V^{15r}$.
The triangle inequality yields
\begin{align*}
\dist(x, V) \leq \abs{x - y} + \dist(y, V) \leq 16r \leq 20r = \la.
\end{align*}
The other equality in \eqref{AuBuXY}
follows from \eqref{AuBu}, which yields
$(H_u)^{-r} \cup X \supset U^{2r} \cup X = \Tn$.
\qedhere\end{proof}

Our goal is to construct admissible faceted test functions
that correspond to facets $(\Omega_{-,i}, \Omega_{+,i})$, $i = u, v$,
which are
in general position with respect to $u$ and $v$ at $(\hat x, \hat t)$.
Therefore, let $\psi_u$, $\psi_v$ be the support functions
constructed
in Proposition~\ref{pr:support-function}
for smooth pairs $(\Omega_{-,u}, \Omega_{+,u})$
and $(\Omega_{-,v}, \Omega_{+,v})$, respectively. 
We define 
\begin{align*}
\vp_u(x,t) &= \la_u \bra{\psi_u(x)}_+ - \mu_u \bra{\psi_u(x)}_- + \al + S(t,\hat s) - S(\hat t, \hat s),\\
\vp_v(x,s) &=  \la_v \bra{\psi_v(x)}_+ - \mu_v \bra{\psi_v(x)}_- + \be + S(\hat t, \hat s) - S(\hat t, s),
\end{align*}
where $[q]_\pm := \max( \pm q, 0)$.

\begin{lemma}
\label{le:order-with-shifts}
There are positive constants $\la_u$, $\mu_v$, chosen large enough,
and positive constants $\mu_u$, $\la_v$, chosen small enough,
so that
\begin{align*}
&u(x,t) \leq \vp_u(x - h, t)
    &&\text{for } x \in \Tn,\ t \in [\hat t - \tau, \hat t + \tau],
        \ h\in \cl B_r(0),
\intertext{and}
&\vp_v(x-h,t) \leq v(x,t)
    &&\text{for } x \in \Tn,\ t \in [\hat s - \tau, \hat s + \tau],
        \ h\in \cl B_r(0).
\end{align*}
\end{lemma}

\begin{proof}
We shall show the statement in the case of $\vp_u$.

Let us define $\vp^r(x,t) = \min_{\abs{h} \leq r} \vp_u(x -h, t)$.
It is easy to see that $\vp^r \in C(\cl Q)$
and $\vp^r(\cdot, t) - g(t)$
is a support function of
the (not necessarily smooth) pair $((H_v)^{-r}, (G_u)^{-r})$.
Moreover, the statement of the lemma is clearly equivalent to
$u \leq \vp^r$ for $x\in \Tn$, $t \in [\hat t - \tau, \hat t + \tau]$.

The space $\Tn$ can be divided into three parts
with the help of observation \eqref{AuBuXY}:

\begin{enumerate}
\item $M_1 = (X \cup V^\la)^c$.
    Since $\cl{M_1} \subset (\cl G_u)^{-2r} \subset (G_u)^{-r}$
        by \eqref{AuBuXY},
        $\psi_u > 0$ in $G_u$ and $(G_u)^{-r}$ is open,
      continuity yields 
      $\min_{\abs{h} \leq r} \min_{\cl{M_1}+h} \psi_u > 0$
      and we can choose $\la_u$ large enough to satisfy $u \leq \vp^r$
      in $\cl{M_1} \times [\hat t - \tau, \hat t + \tau]$
      due to boundedness of $u$.

\item $M_2 = (H_u)^{-r} \setminus M_1 = (H_u)^{-r} \cap (X \cup V^\la)$.
      Observe that $\psi_u \geq 0$ on $H_u$
      by construction.
      Therefore $u \leq \vp^r$
      on $M_2 \times [\hat t - \tau, \hat t + \tau]$
      due to \eqref{lambda-extened-ordering}.

\item $M_3 = \pth{(H_u)^{-r}}^c$.
      We can choose $\mu_u$ small enough so that 
\begin{align*}
- \mu_u \norm{\psi_u}_\infty 
    \geq \max_{X \times [\hat t -\tau, \hat t + \tau]}
            u(x,t) - a - S(t, \hat s) + S(\hat t, \hat s)
\end{align*}
because the right-hand side is negative
thanks to \eqref{u-upper-strict-bound}.
Since $M_3 \subset X$ by \eqref{AuBuXY},
we see that $u \leq \vp^r$ with this choice of $\mu_u$
on $M_3 \times [\hat t -\tau, \hat t + \tau]$.
\end{enumerate}

The result for $\vp_v$ can be obtained similarly.
\qedhere\end{proof}

Let us recall the definition of smooth pairs
$(\Omega_{-,i}, \Omega_{+,i})$
in \eqref{def-of-facets}.
Now, it is clearly true, by definition,
that $\vp_u(\hat x, \hat t) = u(\hat x, \hat t)$
and $\vp_v(\hat x, \hat s) = v(\hat x, \hat s)$.
Next,
we give a simple argument to show that
$\cl B_r(\hat x) \subset H_u \setminus \cl G_u$
and $\cl B_r(\hat x) \subset H_v \setminus \cl G_v$.
This is immediate,
by recalling the obvious fact $\hat x \in U \cap V$ 
and its consequence
$\cl B_r(\hat x) \subset U^r \cap V^r \subset H_u \cap H_v$,
which, together with \eqref{dist-of-facets}, yield the claim.
Consequently, by construction, for all $\abs{z} \leq r$,
$\vp_u$ is an admissible faceted test function at $(\hat x + z, \hat t)$
with smooth pair
$(\Omega_{-,u}, \Omega_{+,u})$
and $\vp_v$ is an admissible faceted test function at
$(\hat x + z, \hat t)$ with smooth pair
$(\Omega_{-,v}, \Omega_{+,v})$.

Therefore Lemma~\ref{le:order-with-shifts}
yields that
$\vp_u$ is in general position of radius $r$ with
respect to $u$ at $(\hat x, \hat t)$,
and $-\vp_v$ is in general position of radius $r$
with respect to $-v$ at $(\hat x, \hat s)$.

By definition of viscosity solutions,
we must have for some $\de \in (0, r)$
\begin{align*}
S_t(\hat t,\hat s) + F(0,
    \essinf_{B_\de(\hat x)} \Lambda(\Omega_{-,u}, \Omega_{+,u})) \leq 0 
\end{align*}
and
\begin{align*}
-S_s(\hat t,\hat s)
    + F(0, \esssup_{B_\de(\hat x)}
        \Lambda(\Omega_{-,v}, \Omega_{+,v})) \geq 0. 
\end{align*}
Subtracting these two inequalities, we get
\begin{align*}
0 < \frac{\gamma}{(T - \hat t)^2} + \frac{\gamma}{(T - \hat s)^2}
\leq F(0, \esssup_{B_\de(\hat x)}
        \Lambda(\Omega_{-,v}, \Omega_{+,v}))
- F(0, \essinf_{B_\de(\hat x)} \Lambda(\Omega_{-,u}, \Omega_{+,u})).
\end{align*}
Since $\Omega_{\pm,u} \subset\Omega_{\pm,v}$,
the facets are ordered and
the comparison principle Theorem~\ref{th:monotonicity} implies 
\begin{align*}
\essinf_{B_\de(\hat x)} \Lambda(\Omega_{-,u}, \Omega_{+,u})
\leq \esssup_{B_\de(\hat x)} \Lambda(\Omega_{-,v}, \Omega_{+,v}),
\end{align*}
and we get a contradiction due to
the degenerate ellipticity of $F$.

\subsection{Case II}

Since this case is quite classical,
similar to \cite[7.\S G]{GG98ARMA},
we just give a short outline
of the proof.

Let us first recall that
there exist $\e, \si, \ga$, $\zeta$, $\abs\zeta < \kappa(\e)$,
and $(z, z') \in \argmax_{\cl Q \times \cl Q} \Phi_\zeta$
such that $x - y \neq \zeta$.

Using the special structure of $\Phi_\zeta$,
we can proceed as in the proof of \cite[Theorem~3.1]{OhnumaSato97},
using the elliptic result \cite[Theorem~3.2]{CIL}
that is then extended to the parabolic result in
\cite[Corollary~3.6]{OhnumaSato97},
and
we obtain symmetric matrices $X, Y \in \mathcal{S}^n$, $X \leq Y$,
such that
\begin{align*}
    (\hat\Phi_{\zeta,t}, \hat\Phi_{\zeta,x}, X)
            &\in \cl{\mathcal{P}^{2,+}}(u(\hat x, \hat t)),
            &
    (-\hat\Phi_{\zeta,s}, -\hat\Phi_{\zeta,y}, Y)
            &\in \cl{\mathcal{P}^{2,-}}(v(\hat y, \hat s)).            
\end{align*}
Here $\cl{\mathcal{P}^{2,\pm}}$ are the closures of
the second order parabolic semijets, as defined in \cite[\S8]{CIL}.

Let us observe that
\begin{align*}
p := 
    \hat\Phi_{\zeta,x} = -\hat\Phi_{\zeta,y}
    = \frac{\hat x - \hat y - \zeta}{\e}
    \neq 0,
\end{align*}
where $p \in \Rn$ has to be interpreted
as the element of the equivalency class
$\hat x - \hat y - \zeta + \Z^n$ with the smallest norm.

Therefore, the definition of the viscosity solution
implies
\begin{align*}
\hat\Phi_{\zeta,t} + F(p, k(p, X)) &\leq 0,\\
-\hat\Phi_{\zeta,s} + F(p,k(p, Y)) &\geq 0.
\end{align*}
Summing the inequalities, we arrive at
\begin{align*}
    \hat\Phi_{\zeta,t} + \hat\Phi_{\zeta,s}
    \leq 
    F(p, k(p, Y)) - F(p, k(p, X)).
\end{align*}
The right-hand side is nonpositive due to
the ellipticity of $F$, $X \leq Y$ and
the formula \eqref{Lambda-nondeg}.
We have
\begin{align*}
0< \frac{\gamma}{(T - \hat t)^2} + \frac{\gamma}{(T - \hat s)^2} \leq 0,
\end{align*}
a contradiction.

\bigskip
The proof of the comparison principle, Theorem~\ref{th:comparison},
is now complete. \qed

\section{Existence of solutions by stability}
\label{sec:existence}

We shall prove the existence of solutions of \eqref{tvf}
via an approximation by parabolic problems
\begin{align}
\label{approximate-problem}
\begin{cases}
    u_t + F(\nabla u, -\partial^0 E_m(u(\cdot, t))) = 0,\\
    \at{u}{t=0}= u_0,
\end{cases}
\end{align}
where $E_m$ are some smooth energies approximating $E$.
For given $m> 0$,
we choose the energy $E_m$
as a smooth approximation of the total variation 
energy,
given as
\begin{align*}
E_m(\psi) :=
\begin{cases}
\int W_m(\nabla \psi) \dx &
    \text{for $\psi \in H^1(\T^n)$},\\
+\infty & \text{for } \psi \in L^2(\T^n) \setminus H^1(\T^n),
\end{cases}\\
\text{where} \quad
W_m(p) := (\abs{p}^2 + m^{-2})^{1/2} + m^{-1} \abs{p}^2.
\end{align*}
The main features of the approximation are summarized
in the following proposition.
Let us point out that other approximations are possible,
as long as $W_m$ form a decreasing sequence
of smooth, uniformly convex, radially symmetric functions
converging to $W(p) = \abs{p}$ pointwise as $m\to\infty$.

\begin{proposition}
\label{pr:Em-properties}
\begin{inparaenum}[(a)]
\item
$E_m$ form a decreasing sequence of
proper convex lower semi-continuous functionals
and
$E = \pth{\inf_m E_m}_*$,
the lower semi-continuous envelope of $\inf_m E_m$ in $L^2(\Tn)$.

\item
The subdifferential $\partial E_m$ 
is a singleton for all $\psi \in \mathcal{D}(\partial E_m) = H^2(\T^n)$
and its canonical restriction can be expressed as
\begin{align}
\label{approximate-operator}
    -\partial^0 E_m(\psi)
    = \divo \bra{(\nabla_p W_m)(\nabla \psi)}
    = \divo \bra{
        \frac{\nabla \psi}{(\abs{\nabla \psi}^2 + m^{-2})^{1/2}}} 
        + m^{-1} \Delta \psi
        \quad \text{a.e.}
\end{align}

\item
Due to the ellipticity of $F$,
if $-\partial E_m(\psi)$ is interpreted as \eqref{approximate-operator},
the problem \eqref{approximate-problem} is a degenerate parabolic
problem that has a unique global viscosity solution
for given continuous initial data $u_0 \in C(\T^n)$.
\end{inparaenum}
\end{proposition}

\begin{proof}
Clearly $W_m$ is a sequence of
decreasing functions,
and hence $E_m$ is a decreasing sequence of functionals.
The proof of lower semi-continuity and convexity of $E_m$,
together with the characterization of the subdifferential in (b),
can be found in \cite{Evans}.

To prove $\pth{\inf_m E_m}_* = E$, let us denote $F = \inf_m E_m$.
We need to show that
$E(\psi) = F_*(\psi) = \liminf_{\psi_k \to \psi} F (\psi_k)$
for all $\psi \in L^2$.
First, let us assume that $\psi \in H^1(\Tn)$.
Then $E_m(\psi) < \infty$;
moreover, since $W_m(D\psi) \downarrow W(D\psi)$ a.e.,
the monotone convergence theorem implies
$E(\psi) = \inf E_m(\psi) = F(\psi)$.
On the other hand, if $\psi \notin H^1$ then $E_m(\psi) = \infty$
for all $m$ and hence $F(\psi) = \infty$.
We therefore conclude that $E \leq F$,
and, recalling that $E$ is l.s.c. in $L^2$, we also see that $E \leq F_*$.
Let now $\psi \in BV(\Tn)$.
There exists a sequence $\psi_k \in C^\infty(\Tn) \subset H^1(\Tn)$
such that $\psi_k \to \psi$ in $L^2$
and $F(\psi_k) = E(\psi_k) \to E(\psi)$;
see \cite[Theorem B.3]{ACM04}.
Therefore 
\begin{equation*}
E(\psi) \leq F_*(\psi)
    \leq \liminf_k F(\psi_k) = \liminf_k E(\psi_k) = E(\psi),
\end{equation*}
which concludes the proof of (a).

For (c) see \cite{CIL}.
\qedhere\end{proof}

In the following two subsections
we prove that the limit of solutions of \eqref{approximate-problem}
is the unique viscosity solution of \eqref{tvf}
with the same initial data.
First, we establish a key stability property of subsolutions
and supersolutions of \eqref{approximate-problem},
which we then follow by a construction of barriers near the initial data
that prevent the formation of an initial boundary layer
in the limit.

\subsection{Stability}

Let us recall the definition of the half-relaxed limits
\begin{align*}
\halflimsup_{m\to\infty} u_m(x,t) &:= 
    \lim_{k\to\infty} \sup_{m \geq k} \sup_{\substack{\abs{y - x} \leq \ov k\\\abs{s - t} \leq \ov k}} u_m(y,s), \\ 
    \halfliminf_{m\to\infty} u_m(x,t) &:= 
    \lim_{k\to\infty} \inf_{m \geq k} \inf_{\substack{\abs{y - x} \leq \ov k\\\abs{s - t} \leq \ov k}} u_m(y,s).
\end{align*}
We have the following stability property.

\begin{theorem}[Stability]
\label{th:stability}
Let $u_m$ be a sequence of subsolutions of \eqref{approximate-problem}
on $\Tn \times [0,\infty)$,
and let $\ou = \halflimsup_{m\to\infty} u_m$.
Assume that $\ou < +\infty$ in $\Tn \times [0,\infty)$.
Then $\ou$ is a subsolution of \eqref{tvf}.

Similarly, $\underline u = \halfliminf_{m\to\infty} u_m$
is a supersolution of \eqref{tvf}
provided  that $u_m$ is a sequence of supersolutions of
\eqref{approximate-problem} and $\underline u > -\infty$.
\end{theorem}

Before proceeding with the proof of the stability theorem,
let us first give a motivation and a brief outline.
Intuitively, the main challenge in the proof of stability
lies in the different natures of the nonlocal limit problem \eqref{tvf}
and the local approximate problems \eqref{approximate-problem}.

The basic idea of the standard stability proof
in the theory of viscosity solutions is the observation
that if $\ou - \vp$ has a strict local maximum at a point
$(\hat x, \hat t)$,
then $u_m - \vp$ will have a local maximum at a nearby point
$(x_m, t_m)$,
for a subsequence of $m$.
Using the knowledge that $u_m$ is a subsolution of
the approximate problem,
classical maximum principle arguments provide
information on $\vp$ at $(x_m,t_m)$, and
the proof is concluded by taking the limit $m \to \infty$.

However, the local nature of the approximate problem
allows us to recover only local information 
about the test function $\vp$ at $(\hat x, \hat t)$.
Therefore, in the current setting, this direct argument works only
as long as $\nabla \vp(\hat x, \hat t) \neq 0$.

When $\vp$ is an admissible faceted test function
at $(\hat x, \hat t)$,
in particular $\nabla \vp = 0$ in the neighborhood of $(\hat x, \hat t)$,
the viscosity condition on $\ou$ is formulated
via the inherently nonlocal quantity
$\Lambda(\Omega_-, \Omega_+)(\hat x)$.
Fortunately,
it is possible to encode the missing nonlocal information
into the test function itself by means of a small perturbation
of $\vp$ depending on $m$ (and other extra parameters),
as long as the perturbation can be shown to vanish uniformly
when taking the various limits.
The perturbed test function method was pioneered
in the theory of viscosity solutions by Evans \cite{Evans89},
where it was applied to the homogenization of an elliptic
problem and thus the choice of perturbation was
completely different.

The perturbation $\vp_{a,m}(x,t) = f_{a,m}(x) + g(t)$
of $\vp(x,t) = f(x) + g(t)$
is motivated by the interpretation of
$\divo (\nabla f/\abs{\nabla f})$
as the $L^2$-gradient of the total variation energy
in the gradient flow \eqref{2.4} with initial data $f$.
The evolution of the gradient flow can be approximated
by the implicit Euler scheme with finite time-step $a > 0$,
whose first time-step is in fact the solution $f_a$ of the resolvent problem
\eqref{resolvent-problem}.
The resolvent problem is, however, conveniently approximated
by the solution $f_{a,m}$ of the resolvent problem for energy $E_m$.
Moreover, the nonlocal nature of resolvent problems
provides the missing information about the shape of the facet of $\vp$.

We collect the relevant features of the approximation by resolvent problems
in Proposition~\ref{pr:resolvents-regularity} below.
The statement contains an extra parameter $\e$
which will be used throughout the proof of stability
to rigorously justify that the limit problem does not depend
on $\nabla\vp$ outside of the facet.

The proof of stability for subsolutions
proceeds roughly in the following manner:
\begin{inparaenum}[1)]
\item
as test functions with $\nabla \vp \neq 0$ can be handled
by the classical stability proof,
we restrict our attention to admissible faceted test functions
$\vp(x,t) = f(x) + g(t)$
such that $\ou - \vp$ has a maximum at $(\hat x,\hat t)$ on a facet;
\item
we show that it is possible to
assume that the set of points of maxima can be restricted
to a small neighborhood of the facet;
this is a consequence of $\vp$ being in general position
with respect to $u$;
\item
we perturb the test function by solving
the resolvent problems for energy $E_m$ 
with ``time-step'' $a$ and right-hand side $\e f$;
\item
after using the classical maximal principle arguments
to infer information about the action of operator 
\eqref{approximate-operator} on $\vp_{a,m}$,
we sent $m \to \infty$, which provides
information about the action of the operator \eqref{tvf}
on $\vp_a$;
finally,
\item
sending $a \to 0$ and $\e \to 0$
together with a geometric argument concludes the proof.
\end{inparaenum}

\begin{proposition}
\label{pr:resolvents-regularity}
For $f \in Lip(\Tn)$ and positive constants $m$, $a$, $\e$, 
the resolvent problems
\begin{subequations}
\begin{align}
\label{E-resolvent}
f_a + a \partial E(f_a) &\ni \e f,\\
\label{smooth-resolvent}
f_{a,m} + a \partial E_m(f_{a,m}) & \ni \e f,
\end{align}
\end{subequations}
admit unique Lipschitz continuous solutions $f_a$ and $f_{a,m}$, respectively, and
\begin{align}
\label{lipschitz-bound-stability}
\norm{\nabla f_a}_\infty, \norm{\nabla f_{a,m}}_\infty
    \leq \e \norm{\nabla f}_\infty.
\end{align}
Moreover, $f_{a,m} \in C^{2,\al}(\T^n)$ for some $\al > 0$.

In what follows, let $\e > 0$ be fixed and we define
\begin{align}
\label{def-of-ha}
h_a &:= \frac{f_a - \e f}{a}, &
h_{a,m} &:= \frac{f_{a,m} - \e f}{a} = -\partial^0 E_m(f_{a,m}).
\end{align}
Then, for fixed $a > 0$,
\begin{align*}
    f_{a,m} &\rightrightarrows f_a && \text{uniformly as $m\to\infty$,
             and}\\
    h_{a,m} &\rightrightarrows h_a && \text{uniformly as $m \to \infty$.}
\intertext{Additionally,}
    f_a &\rightrightarrows \e f && \text{uniformly as $a \to 0$}.\\
\intertext{If furthermore $f \in \mathcal{D}(\partial E)$, then also}
    h_a &\to -\partial^0 E(\e f) && \text{in $L^2(\Tn)$ as $a\to0$}.
\end{align*}
\end{proposition}

\begin{proof}
The existence and uniqueness of $f_{a,m}$ and $f_a$ in $L^2(\Tn)$
is straightforward
from the theory of monotone operators;
see Section~\ref{sec:resolvent-equation}.
$C^{2,\al}$ regularity of $f_{a,m}$
follows from the elliptic regularity theory
since $I + a\partial E_m$ is a quasilinear uniformly elliptic operator;
see \cite{CaffarelliCabre}.
The Lipschitz bound on $f_{a,m}$
is obtained by a standard comparison argument.
Indeed,
since $f_{a,m}$ is a classical solution of \eqref{smooth-resolvent}
with right-hand side $\e f$
and the problem is translation-invariant,
we have that $f_{a,m}(\cdot + z) - \e \norm{\nabla f}_\infty \abs{z}$
is a subsolution of \eqref{smooth-resolvent}
with right-hand side
$\e(f(\cdot+z) -  \norm{\nabla f}_\infty \abs{z}) \leq \e f$
for any $z \in \Tn$.
Therefore the comparison principle yields
$f_{a,m}(x + z) - \e \norm{\nabla f}_\infty \abs{z} \leq f_{a,m}(x)$
for all $x,z \in \Tn$,
which implies the Lipschitz bound
$\norm{\nabla f_{a,m}}_\infty \leq \e \norm{\nabla f}_\infty$.

As we shall show,
the approximation of the $E$-resolvent problem \eqref{E-resolvent}
by the $E_m$-resolvent problems \eqref{smooth-resolvent}
is a consequence of the \emph{Mosco convergence} of $E_m$ to $E$.
Here we present only the outline of the argument
and we refer the reader to the book of Attouch \cite{Attouch}
for the definition and further details.
Indeed, due to Proposition~\ref{pr:Em-properties}(a),
\cite[Theorem~3.20]{Attouch} implies that
$E_m$ Mosco-converges to $E$ as $m \to \infty$.
Consequently, the Mosco convergence of $E_m$ to $E$
implies the graph convergence of $\partial E_m$ to $\partial E$
as $m\to\infty$;
see \cite[Theorem~3.66]{Attouch}.
Finally, the graph convergence is equivalent to
the resolvent convergence: we have, for fixed $a > 0$,
$f_{a,m} \to f_a$ in $L^2$ as $m \to \infty$;
see \cite[Theorem~3.62]{Attouch}.
We conclude by recalling the uniform Lipschitz bound 
$\norm{\nabla f_{a,m}}_\infty \leq \e \norm{\nabla f}_\infty$
and its consequence
$\norm{\nabla h_{a,m}}_\infty \leq 2\e \norm{\nabla f}_\infty/a$,
which yield the uniform convergence of $f_{a,m}$ and $h_{a,m}$
as $m \to\infty$.
In particular, we recover the Lipschitz bound
$\norm{\nabla f_a}_\infty \leq \e\norm{\nabla f}_\infty$.

By the classical theory of resolvent problems,
$f_a \to \e f$ in $L^2(\Tn)$; see \cite{Evans}.
This yields, together with the uniform Lipschitz bound
\eqref{lipschitz-bound-stability},
the uniform convergence.
Additionally, convergence of $h_a$ is a restatement of
Proposition~\ref{pr:resolvent-convergence}
in light of Remark~\ref{re:support-scal-invar}.

\qedhere\end{proof}

\begin{proof}[Proof of Stability Theorem~\ref{th:stability}]
We show the statement for subsolutions.
To simplify the notation slightly, let us denote $u = \ou$.

Clearly $u_m$ is an upper semi-continuous function.
We must verify that (i) and (ii) of Definition~\ref{def:Lambda-subsol}
hold.

\medskip
\noindent
\textbf{(i)} Let $\vp(x,t) = f(x) + g(t)$ be an admissible test function
in a general position of radius $\eta > 0$ at $(\hat x, \hat t)$
with respect to $u$.
By definition,
$f$ is the support function of some smooth pair $(\Omega_-, \Omega_+)$
with Lipschitz constant $L := \norm{\nabla f}_\infty < 0$.
Observe that the definition also ensures $f(\hat x - z) =0$
for all $\abs{z} \leq \eta$.
We set $\de := \eta/4$.

Let us introduce the closed sets
\begin{align*}
U &= \set{u(\cdot, \hat t) \geq u(\hat x, \hat t)},
    & Z &= \set{f \leq 0}, & N_\rho
    &= U^\rho \cap Z^\rho \quad\text{for } \rho > 0.
\end{align*}
The situation at time $\hat t$ is depicted in
Figure~\ref{fig:stability-geometry}.
\begin{figure}
\centering
\fig{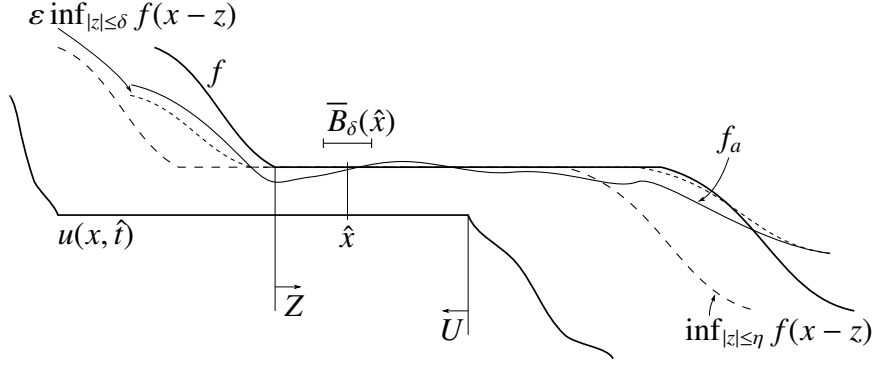}{4.5in}
\caption{Geometry of the graphs around the contact point $(\hat x, \hat t)$,
at time $\hat t$,
in the proof of Theorem~\ref{th:stability}.}
\label{fig:stability-geometry}
\end{figure}

We will fix a number $\e \in (0,1)$ throughout the proof.
The following observation is an important consequence
of the general position of $\vp$ and $u$,
but we will postpone its proof
until after the end of the proof of stability.

\begin{lemma}
\label{cl:strictish-maximumum}
Assume that $\vp$ is in a general position of radius $\eta$
with respect to $u$.
Then,
by adding the term $\abs{t - \hat t}^2$ to $g(t)$ if necessary,
we have
\begin{align}
\label{strict-stability}
u(x,t) - \e \inf_{\abs{z} \leq \de} f(x-z) - g(t)
    < u(\hat x, \hat t) - g(\hat t)
\end{align}
whenever 
        $(x,t) \in \pth{N_\eta \times [\hat t - \eta, \hat t + \eta]}
         \setminus
            \pth{N_\de \times \set{\hat t}}$,
where $\de = \eta / 4$.
\end{lemma}

\medskip
Now we recall the definitions of $f_a$ and $f_{a,m}$,
together with $h_a$ and $h_{a,m}$,
and their properties from Proposition~\ref{pr:resolvents-regularity},
which we will use in the rest of the proof.

For given $a > 0$ and $z$, $\abs{z} \leq \de$,
we define the set of maxima
\begin{align*}
A_z := \argmax_{N_{3\de} \times [\hat t - 2\de, \hat t + 2\de]}
    \bra{u(x,t)- f_a(x - z) - g(t)}.
\end{align*}
The uniform convergence $f_a \rightrightarrows \e f$
and the strict ordering in Lemma~\ref{cl:strictish-maximumum}
guarantee the existence of $a_0 > 0$ such that
whenever $a < a_0$ then
\begin{align}
\label{Az-def}
\emptyset &\neq A_z
    \subset N_{2\de} \times [\hat t - \de, \hat t + \de]
    &&\text{for all } \abs{z} \leq \de.
\end{align}
We fix one such $a < a_0$ and find $z$, $\abs{z} \leq \de$,
such that
\begin{align}
\label{z-fa-min}
f_a(\hat x - z) = \min_{\cl B_\de(\hat x)} f_a.
\end{align}
This choice will guarantee \eqref{claim-ha} to hold.

Again, due to the uniform convergence $f_{a,m} \rightrightarrows f_a$
as $m \to \infty$, there is
$(x_a, t_a) \in A_z$ and a sequence $(x_{a,m}, t_{a,m})$
(for a subsequence of $m$) of local maxima of
\begin{align*}
(x, t) \mapsto u_m(x,t) - f_{a,m}(x-z) - g(t).
\end{align*}
such that $(x_{a,m}, t_{a,m}) \to (x_a, t_a)$ as $m \to \infty$
(along a subsequence).

Since $f_{a,m} \in C^{2,\al}$ and $u_m$ is a visc. subsolution of
\eqref{approximate-problem},
we have
\begin{align}
\label{am-ineq-stability}
\begin{aligned}
g'(t_{a,m}) &+ F(\nabla f_{a,m}(x_{a,m} - z), h_{a,m}(x_{a,m} - z))\\
    &= g'(t_{a,m})
      + F(\nabla f_{a,m}(x_{a,m} - z), -\partial E_m(f_{a,m})(x_{a,m} - z)))
    \leq 0.
\end{aligned}
\end{align}
Now we recall the Lipschitz bound \eqref{lipschitz-bound-stability},
which
yields that there exists $p_a$ such that,
along a subsequence,
$\nabla f_{a,m}(x_{a,m}) \to p_a$,
and $\abs{p_a} \leq \e \norm{\nabla f}_\infty = \e L$.
Consequently, we send $m \to \infty$
in \eqref{am-ineq-stability}
along a subsequence,
and
use the uniform convergence $h_{a,m} \rightrightarrows h_a$
and the continuity of $g'$ and $F$ to conclude that
\begin{align}
\label{shifted-ha}
g'(t_a) + F(p_a, h_a(x_a - z)) \leq 0.
\end{align}

To estimate $h_a(x_a -z)$,
we observe an important geometric consequence of
the choice of $z$ in \eqref{z-fa-min}, and the profile of $u$,
which follows from the general position of $\vp$ with
respect to $u$.
Again, we postpone the proof until after the end of the proof
of stability.

\begin{lemma}
\label{cl:h-at-xa}
We have
\begin{align}
\label{claim-ha}
h_a(x_a - z) \leq h_a(\hat x - z) = \min_{\cl B_\de(\hat x)} h_a.
\end{align}
\end{lemma}

Therefore, by ellipticity of $F$ and by \eqref{shifted-ha},
we arrive at
\begin{align}
\label{min-ha-ineq}
g'(t_a) + F(p_a, \min_{\cl B_\de(\hat x)} h_a)
    \leq g'(t_a) + F(p_a, h_a(x_a - z)) \leq 0.
\end{align}

In the next step, we send $a \to 0$ along a subsequence.
We choose the subsequence $a_l$ in such a way that $p_{a_l} \to \hat p$
with $\abs{\hat p} \leq \e L$
(this is possible by $\abs{p_a} \leq \e L$)
and
$\min_{\cl B_\de} h_{a_l}
    \to \liminf_{a \to0} \min_{\cl B_\de} h_a$
as $l \to \infty$.
Since also $t_{a_l} \to \hat t$,
we send $l \to \infty$ and,
with the help of continuity, we conclude from \eqref{min-ha-ineq}
that
\begin{align}
\label{liminf-min-ha}
g'(\hat t) + F(\hat p, \liminf_{a\to0} \min_{\cl B_\de(\hat x)} h_a)
    \leq 0.
\end{align}
It is time to use the knowledge of the limit $h_a \to -\partial^0 E(\e f)$
in $L^2$;
we remind the reader that $\e f \in \mathcal{D}(\partial E)$,
as was observed in Remark~\ref{re:support-scal-invar}.
Because $\cl B_\de(\hat x)$ is contained in the facet,
$-\partial^0 E(\e f) = \Lambda(\Omega_-, \Omega_+)$ a.e. in $B_\de(\hat x)$
by Proposition~\ref{pr:support-function}.
This yields
\begin{align*}
\liminf_{a\to0} \min_{\cl B_\de(\hat x)} h_a
    \leq \essinf_{B_\de(\hat x)} \Lambda(\Omega_-, \Omega_+).
\end{align*}
One last time, the above inequality, together with the ellipticity of $F$
and with \eqref{liminf-min-ha},
leads to
\begin{align}
\label{essinf-ha}
g'(\hat t) + F(\hat p, \essinf_{B_\de(\hat x)} \Lambda(\Omega_-, \Omega_+))
    \leq 0.
\end{align}

As the last step, we recall that the reasoning is valid for any
fixed $\e \in (0,1)$, yielding $\hat p = \hat p(\e)$.
After sending $\e \to 0$ and realizing that
$\abs{\hat p(\e)} \leq \e L \to 0$,
we recover the inequality \eqref{essinf-ha} with $\hat p = 0$.
This finishes the proof for admissible faceted test functions.

\medskip
\noindent
\textbf{(ii)}
We can follow the standard stability argument
in the theory of viscosity solutions.

\medskip
The proof for supersolutions is completely analogous.
\qedhere\end{proof}

We still owe the reader a proof of Lemma~\ref{cl:strictish-maximumum}
and \ref{cl:h-at-xa},
which both rely on the following geometrical lemma.

\begin{lemma}
\label{le:simple-facet-order}
Suppose that $\vp$ is in general position of radius $\eta$
with respect to $u$ at $(\hat x, \hat t)$,
where $\vp(x,t) = f(x) + g(t)$.
Then
\begin{align*}
u(x,t) \leq g(t) - g(\hat t) + u(\hat x, \hat t) \quad
    \text{for } (x,t) \in Z^\eta \times [\hat t - \eta, \hat t + \eta],
\end{align*}
and
\begin{align*}
f(x) \geq 0 \quad \text{for } x \in U^\eta,
\end{align*}
where $Z := \set{f \leq 0}$
and $U := \set{u(\cdot, \hat t) \geq u(\hat x, \hat t)}$.
\end{lemma}

\begin{proof}
By definition of general position, Definition~\ref{def:general-position},
\begin{align*}
u(x - z, t) - f(x) - g(t) &\leq u(\hat x, \hat t) - f(\hat x) - g(\hat t)
&&\text{for all } \abs{z} \leq \eta.
\end{align*}
We recall that $f(\hat x) = 0$,
and, by definition, whenever $y \in Z^\eta$ there exist $x \in Z$
and $z$, $\abs{z} \leq \eta$, such that $y = x - z$.
Similarly, if $x \in U^\eta$ then there exists $z$, $\abs{z} \leq \eta$,
such that $x - z \in U$.
The result follows. 
\qedhere\end{proof}

\begin{proof}[Proof of Lemma~\ref{cl:strictish-maximumum}]
We shall use the notation from 
the proof of Theorem~\ref{th:stability}.
Let us recall that $\e \in (0,1)$.
For any $t \in [\hat t - \eta, \hat t + \eta]$,
we split the set $N_{2\de}$ into two disjoint parts:

\begin{itemize}
\item
If $x \in N_{4\de} \setminus Z^\delta \subset Z^{4\de} = Z^\eta$,
it is clear that $\inf_{\abs{z} \leq \de} f(x-z) > 0$
and therefore
\begin{align*}
u(x,t) - \e \inf_{\abs{z} \leq \de} f(x-z) - g(t)
    < u(x,t) - g(t) \leq u(\hat x, \hat t) - g(\hat t),
\end{align*}
where the second inequality follows
from Lemma~\ref{le:simple-facet-order} and the fact that $x \in Z^\eta$.

\item
On the other hand,
if $x \in N_{4\de} \cap Z^\delta$
then $\inf_{\abs{z} \leq \de} f(x-z) \leq 0$
and the definition of general position yields
\begin{align*}
u(x,t) - \e \inf_{\abs{z} \leq \de} f(x-z) - g(t)
    \leq  u(x,t)- \inf_{\abs{z} \leq \de} f(x-z) - g(t)
    \leq u(\hat x, \hat t) - f(\hat x) - g(\hat t).
\end{align*}
One readily observes that
\begin{compactenum}
\item
the first inequality is strict whenever
$\inf_{\abs{z} \leq \de} f(x-z) < 0$ (recall that $\e < 1$);
\item
the second inequality is strict whenever
\begin{compactitem}
\item
$t \neq \hat t$ (by adding $\abs{t -\hat t}^2$ to $g(t)$ if necessary).
\item
$t = \hat t$, $\inf_{\abs{z} \leq \de} f(x-z) = 0$ and $x \notin U$.
\end{compactitem}
\end{compactenum}
\end{itemize}
We recover Lemma~\ref{cl:strictish-maximumum} by collecting
all these cases.
\qedhere\end{proof}

\begin{proof}[Proof of Lemma~\ref{cl:h-at-xa}]
Recalling the definition of $h_a$ in \eqref{def-of-ha},
we want to show
\begin{align*}
f_a(x_a - z) - f(x_a - z)
    \leq f_a(\hat x - z) - f(\hat x - z)
    = \min_{\cl B_\de(\hat x)} f_a - f.
\end{align*}
The equality follows directly from \eqref{z-fa-min},
after realizing that $f(\hat x - z) = 0$ since $\hat x -z$
lies in the facet.

To prove the inequality, we make the following two observations:
\begin{enumerate}
\item
$-f(x_a - z) \leq 0 = f(\hat x - z)$:
indeed,
since $x_a \in N_{2\de}$ due to \eqref{Az-def} and the choice of $a$,
we observe that $x_a - z \in N_{3\de} \subset U^\eta$
and therefore Lemma~\ref{le:simple-facet-order} applies.

\item
$f_a(x_a-z) \leq f_a(\hat x - z)$:
as $(x_a, t_a) \in A_z$,
i.e $(x_a, t_a)$ is a point of maximum,
and $\hat x \in U \cap Z \subset N_{3\de}$,
we must have
\begin{align*}
u(x_a, t_a) - f_a(x_a -z) - g(t_a)
    \geq u(\hat x, \hat t) - f_a(\hat x -z) - g(\hat t).
\end{align*}
A rearrangement of this inequality yields
\begin{align*}
f_a(x_a - z) \leq f_a(\hat x - z)
    + \bra{u(x_a,t_a) - g(t_a) - u (\hat x, \hat t) + g(\hat t)},
\end{align*}
and, because $x_a \in Z^\eta$
and thus Lemma~\ref{le:simple-facet-order} implies that
the expression in the bracket is nonpositive.
\end{enumerate}
The inequalities in (a) and (b) together imply Lemma~\ref{cl:h-at-xa}.
\qedhere\end{proof}

\subsection{Existence}

The standard consequence of the stability is the following existence theorem.

\begin{theorem}[Existence]
\label{th:existence}
If $F$ is degenerate elliptic and $u_0 \in C(\Tn)$,
there exists a unique solution
$u \in C(\Tn \times [0, \infty))$ of \eqref{tvf}
with the initial data $u_0$.
Furthermore, if $u_0 \in \Lip(\Tn)$ then
\begin{align*}
\norm{\nabla u(\cdot, t)}_\infty \leq \norm{\nabla u_0}_\infty.
\end{align*}
\end{theorem}

The proof of the theorem will proceed in three steps:
\begin{inparaenum}[1)]
\item due to the stability, by finding the solution $u_m$ of
the problem \eqref{approximate-problem}
for all $m \geq 1$,
we can find a subsolution $\ou$ and a supersolution $\uu$
of \eqref{tvf};
\item \label{existence-step-2}
    a barrier argument at $t = 0$ shows that
$\ou$ and $\uu$ have the correct initial data $u_0$; and
\item \label{existence-step-3}
the comparison principle shows that $\ou = \uu$
is the unique viscosity solution of \eqref{tvf},
and the Lipschitz estimate holds.
\end{inparaenum}

Before giving a proof of the existence theorem,
we construct barriers for step \ref{existence-step-2}.

Since the operator \eqref{approximate-operator}
degenerates at points where $\nabla u = 0$ as $m \to\infty$,
it seems to be necessary to construct barriers that depend
on $m$.
We will use the Wulff functions for energy $E_m$;
these were previously considered in the proof of
stability for general equations of the type \eqref{approximate-problem}
in one-dimensional setting in \cite{GG99CPDE}.

As we did before, let us set
\begin{align*}
W(p) &= \abs{p}, &
    W_m(p) = (\abs{p}^2 + m^{-2})^{1/2} + m^{-1} \abs{p}^2.
\end{align*}
Given these convex functions, we define their 
convex conjugates (via the Legendre-Fenchel transformation)
\begin{align}
\label{legendre-fenchel}
W_m^*(x) = \sup_{p\in\Rn} \bra{p \cdot x - W_m(p)}, \qquad x \in \Rn.
\end{align}
$W^*(x)$ is defined analogously.
We recall that $W(p)$ is the support function of $\cl B_1(0)$
and therefore
\begin{align*}
W^*(x) =
\begin{cases}
0 & \abs{x} \leq 1,\\
+\infty & \abs{x} > 1.
\end{cases}
\end{align*}

Functions $W_m^*$ are the Wulff shapes of the energies $E_m$
in the sense that the operator \eqref{approximate-operator}
applied to $W_m^*$ gives a constant for all $x$.
We summarize the properties of the Wulff function
in the following lemma.

\begin{lemma}
\label{le:Wulff}
The functions $W_m^*$ are:
\begin{inparaenum}[(i)]
\item smooth functions on $\Rn$;
\item strictly convex, radially symmetric;
\item increasing in $m$;
\item $W_m^* \geq -m^{-1}$ with equality only at $x = 0$;
\item $W_m^* \to 0$ uniformly on $\set{\abs{x} \leq 1}$
    as $m \to\infty$;
\item $W_m^* \to +\infty$ uniformly on $\set{\abs{x} \geq R}$
for any $R > 1$ as $m\to\infty$; and finally,
\item 
\begin{align*}
-\partial^0 E_m(W_m^*) \equiv n.
\end{align*}
\end{inparaenum}
\end{lemma}

\begin{proof}
Since $W_m$ is smooth
and $D_p^2 W_p \geq m^{-1} I$,
we have that $W_m$ is in fact uniformly convex
and $W_m(p)/\abs{p} \to \infty$ as $\abs{p} \to \infty$. 
Therefore the supremum in \eqref{legendre-fenchel}
is attained at a unique point $p$ such that $x = \nabla_p W_m(p)$,
and $p$ can be written as a smooth function of $x$
using the implicit function theorem.
(ii) is clear since $W_m$ strictly convex,
radially symmetric;
(iii) holds because $W_m$ are decreasing in $m$;
(iv) follows from \eqref{legendre-fenchel} with $p = 0$;
we observe that $m^{-1} \leq W_m^* < 0$ in $\set{\abs{x} \leq 1}$,
which yields (v);
(vi) can be derived from convexity of $W_m^*$,
\eqref{legendre-fenchel}
and from the fact that $W_m(p) \searrow \abs{p}$ as $m \to \infty$;
finally, (vii) follows from the Fenchel identity
$x = \nabla_p W_m(p)$ if and only if $p = \nabla W_m^*(x)$,
which gives,
using the formula in Proposition~\ref{pr:Em-properties}(b),
$-\partial^0 E_m(W_m^*)(x) = \divo \bra{(\nabla_p W_m)(\nabla W_m^*(x))}
    = \divo x = n.$
\qedhere\end{proof}

Since $F$ depends on the gradient as well,
while the gradient of $W_m^*$ blows up as $m \to 0$,
we consider rescaled cut-offs of these functions
with linear growth at infinity.
Given positive parameters $A$ and $q$,
where $A$ will represent scaling and $q$
the maximal slope,
we introduce a smooth, nondecreasing cut-off function
\begin{align*}
\ta_q(s) = 
\begin{cases}
s & \abs{s} \leq q/2,\\
q \sign s & \abs{s} \geq 2 q.
\end{cases}
\end{align*}
Since $W_m^*$ are smooth, radially symmetric,
we define the rescaled modified Wulff functions on $\Rn$ as
\begin{align*}
w_m(x; A, q) =
    \int_0^{\abs{x}}
        \ta_q\pth{\abs{\nabla W_m^*}(A s x/\abs{x})} \diff s.
\end{align*}

\begin{lemma}
\label{le:wm-properties}
Let $A$ and $q$ be positive constants. Then
$w_m(x; A, q)$ are smooth, convex,
radially symmetric, nonnegative functions on $\Rn$, and
they coincide with $A^{-1} W_m^*(A x) + (A m)^{-1}$ on
the closed ball $\set{\abs{\nabla W_m^*}(A x) \leq q/2}$.
Moreover, $\abs{\nabla w(\cdot; A, q)} \leq q$.
\end{lemma}

Functions $w_m$ will be used to build barriers at $t = 0$
for the existence argument.

\begin{lemma}
\label{le:cutoff-wulff-solutions}
Given positive constants $A$ and $q$,
there exist universal constants $B = B(A, q, F, n)$
and $m_0 = m_0(A, q, F, n)$
such that, for all $m > m_0$,
\begin{align*}
\phi_m(x,t;A, B, q) := B t + w_m(x; A, q)
\end{align*}
is a (classical) supersolution of \eqref{approximate-problem}
in $\Rn \times [0, \infty)$,
and $-\phi_m(x,t;A,B,q)$ is a (classical) subsolution.
\end{lemma}

\begin{proof}
We choose $m_0$ large enough
to guarantee that $\nabla w_m(x) = \nabla W_m^*(Ax)$
in a neighborhood of $x = 0$, uniformly in $m > m_0 = 4/q$.

Let us now explain why $m_0 = 4/q$ is a reasonable choice.
We first observe that
Lemma~\ref{le:Wulff} implies that $\abs{\nabla W_m^*}(A x)$
is an increasing function of $\abs{x}$,
and that $-(Am)^{-1} \leq A^{-1} W_m^*(A x) < 0$
in $\set{\abs{x} \leq A^{-1}}$.
From this, and the choice of $m_0 = 4/q$, we can conclude that
$\abs{\nabla W_m^*}(A x) \leq (Am)^{-1}/(A^{-1}/2) = 2 m^{-1} \leq q/2$
in $\set{\abs{x} \leq A^{-1}/2}$.
Finally, Lemma~\ref{le:wm-properties} and the chain rule yield that
$\nabla w_m(x) = \nabla W_m^*(Ax)$
for all $\abs{x} \leq A^{-1}/2$ and $m \geq m_0$.

To find $B$,
we split $\Rn$ into two sets:
\begin{compactitem}
\item if $x$ such that $\abs{\nabla W_m^*(A x)} \leq q/2$:
    then $\nabla w_m(x) = \nabla W_m^*(A x)$ and thus
    we conclude that $-\partial^0 E_m(w_m)(x) = A n$;
\item otherwise: then as observed above, 
    we must have $\abs{x} > A^{-1}/2$ and
    $\abs{\nabla w_m} \geq q/2$,
    which yields $\abs{-\partial E_m(w_m)(x)} \leq C$
    for some constant $C$, $C\geq An$,
    independent of $x$ and $m \geq m_0$;
    this is a straightforward calculation,
    by evaluating the operator \eqref{approximate-operator}
    applied to $w_m$, which is radially symmetric,
    and observing that it can be bounded
    from above, independently of $m$ and $x$,
    provided that $\abs{x} \geq A^{-1}/2$ and $\abs{\nabla w_m} \geq q/2$.
\end{compactitem}
Finally, by construction, $\abs{\nabla{w_m}} \leq q$ in $\Rn$.
Since $F$ is continuous,
\[B := \max_{\abs{p} \leq q, \abs{s} \leq C} \abs{F(p, s)} \leq \infty.\]
$\phi_m$ are smooth and we have
\begin{align*}
(\phi_m)_t + F(\nabla \phi_m, -\partial E_m(\phi_m)) \geq
    B - B =0.
\end{align*}
From this we see that $\phi_m$ is a classical supersolution
of \eqref{approximate-problem}.
Since $-\partial E_m$ is odd, the same computation shows that
$-\phi_m$ is a subsolution.
\qedhere\end{proof}

Let us conclude by finishing the proof of the existence theorem.

\begin{proof}[Proof of Theorem~\ref{th:existence}]
\emph{Step 1.}
For every $m\geq1$, there exists a unique solution $u_m$
of \eqref{approximate-problem}
with initial data $u_0$.

Since $u_0$ is bounded, $u_m$ are uniformly bounded,
by comparison principle,
on sets $\Tn \times [0, T]$ for every $T > 0$.
Therefore the following half-relaxed limits are bounded:
\begin{align*}
\overline{u} &:= \halflimsup_{m\to\infty} u_m,&
\underline{u} &:= \halfliminf_{m\to\infty} u_m.
\end{align*}
Additionally, by Theorem~\ref{th:stability}, $\overline{u}$
is a supersolution and $\underline{u}$ is a subsolution of \eqref{tvf}.
Furthermore,
\begin{align*}
\underline{u}(x,0) \leq u_0(x) \leq \overline{u}(x,0).
\end{align*}

\emph{Step 2.}
We need to show that in fact the above inequality is an equality.
To achieve that, we will construct a class of radially symmetric
sub- and supersolutions of \eqref{approximate-problem} for all $m$,
arbitrarily close to the initial data.

We set $K := 2 \norm{u_0}_\infty < \infty$.
Recall that $\Tn$ is compact,
hence $u_0$ has a uniform modulus of continuity.
For any fixed $\de > 0$, there exists $r \in (0,1/8)$
such that 
\begin{align}
\label{uniform-cont-u0}
\abs{u_0(x) - u_0(\xi)} < \de
    &&\text{for all } \xi, x \in \Tn,\ \abs{x - \xi} \leq 4r.
\end{align}
We fix $\de > 0$ and set $A = r^{-1}$, $q = 2Kr^{-1}$.
By Lemma~\ref{le:Wulff}(vi),
there exists $\tilde m_0$ such that $A^{-1} W_m^*(A x) \geq K$
for all $x \in \Rn$, $\abs{x} \geq 2r$, and $m > \tilde m_0$.
We choose $m_0 \geq \tilde m_0$ so that the
conclusion of Lemma~\ref{le:cutoff-wulff-solutions} holds
with a universal constant $B$.

We claim that 
\begin{align}
\label{barrier-geq-K}
\phi_m(x, 0; A, B, q) &= w_m(x; A, q) \geq K
&&\text{for $\abs{x} \geq 3r$}.
\end{align}
Indeed, by convexity, $\abs{\nabla W_m^*}(Ax)$
and $\abs{\nabla w_a(\cdot; A, q)}(x)$ are
increasing functions of $\abs{x}$. If $\abs{\nabla W_m^*}(Ax)\leq q/2$
at $\abs{x} = 2r$ then
$w_A(y; A, q) \geq w_A(x; A, q) = A^{-1} W_m^*(A x) + (Am)^{-1} \geq K$
for $\abs{x} = 2r$ and $\abs{y} \geq 2r$.
On the other hand, if $\abs{\nabla W_m^*}(Ax)> q/2$
at $\abs{x} = 2r$,
then $w_A(x; A, q) \geq \frac{q}2 (\abs{x} - 2r)$ for $\abs{x} \geq 2r$,
and the claim is proved since
$\frac{q}2 ( \abs{x} - 2r) \geq \frac{qr}2 = K$ for $\abs{x} \geq 3r$.

For arbitrary $x, \xi \in \Tn$,
we set
\begin{align*}
\overline{\phi}_m^\xi (x,t)
    &:= - \mathring\phi_m(x - \xi,t; A, B, q) + u_0(\xi) - \de, \\
\underline{\phi}_m^\xi (x,t)
    &:= \mathring \phi_m(x - \xi,t; A, B, q) + u_0(\xi) + \de,
\end{align*}
where $\mathring \phi_m$ is the ``periodization'' of $\phi_m$, i.e.,
\begin{align*}
\mathring \phi_m(x, t; A, B, q)
    = \inf_{z \in \Z^n} \phi_m(x + z, t; A, B, q),
    \qquad x \in \Rn, t \geq 0.
\end{align*}
We remark that $\mathring\phi_m$ is still a viscosity supersolution
of \eqref{approximate-problem},
due to the stability properties of viscosity solutions. 
Since $\phi(x,0; A,B, q) \geq 0$, $4r \leq 1/2$,
and \eqref{barrier-geq-K} and
\eqref{uniform-cont-u0} hold,
we have
\begin{align*}
\overline{\phi}_m^\xi(\cdot, 0) \leq u_0
    \leq \underline{\phi}_m^\xi(\cdot, 0)
    \qquad \text{in $\Tn$ for $\xi \in \Tn$, $m > m_0$}.
\end{align*}
Therefore, by the parabolic comparison principle,
\begin{align}
\label{comp-barriers-at-0}
\overline{\phi}_m^\xi \leq u_m
    \leq \underline{\phi}_m^\xi
        \qquad \text{in $\Tn \times [0,\infty)$,
            for $\xi \in \Tn$, $m > m_0$}.
\end{align}
Recalling that
\begin{align*}
\overline \phi_m^\xi(x,t) \geq -B t + u_0(\xi) - \de - (Am)^{-1}
    \qquad \text{for $\abs{x - \xi} \leq r$,}
\end{align*}
we have
${\halfliminf_{m\to\infty} \overline \phi_m^\xi}(\xi, 0) = 
    u_0(\xi) - \de$.
    
We can similarly compute
${\halflimsup_{m\to\infty} \underline \phi_m^\xi}(\xi, 0)
    = u_0(\xi) + \de$
and we recover from \eqref{comp-barriers-at-0}
\begin{align*}
u_0(\xi) -\de \leq \uu(\xi, 0) \leq \ou(\xi, 0) \leq u_0(\xi) + \de
\qquad \text{for $\xi \in \Tn$}.
\end{align*}
Since $\de > 0$ was arbitrary,
we have $\uu(\cdot, 0) = \ou(\cdot, 0) = u_0$.

\emph{Step 3.}
The comparison theorem~\ref{th:comparison}
then yields $\uu = \ou$ in $\Tn \times [0, \infty)$,
and it is the unique solution of \eqref{tvf} with
initial data $u_0$.

To show the Lipschitz estimate, 
we use the standard comparison argument.
Suppose that $u_0 \in \Lip(\Tn)$,
and set $L = \norm{\nabla u}_\infty$,
and let $u$ be the unique viscosity solution.
Since the operator $F$ does not depend on neither $x$ nor $u$,
we observe that $u(x + h) + L\abs{h}$ is a viscosity supersolution
for any $h \in \Tn$.
Therefore, by comparison,
\[
u(x,t) \leq u(x +h,t) + L \abs{h}
    \qquad \text{for } x, h \in \Tn,\ t \geq 0.
\]
The Lipschitz estimate follows.
\qedhere\end{proof}

The proof of the existence theorem, Theorem~\ref{th:existence},
also gives the following approximation result.

\begin{theorem}
\label{th:full-stability}
Let $u_0 \in C(\Tn)$ and let $u_m$
be the unique viscosity
solutions of \eqref{approximate-problem}
on $\Tn \times (0,\infty)$ for $m \geq 1$.
Then $u_m$ converge locally uniformly
on $\Tn \times [0, \infty)$ as $m \to \infty$
to the unique viscosity solution $u$
of \eqref{tvf} with initial data $u_0$.
\end{theorem}

\begin{proof}
It was shown in the proof of Theorem~\ref{th:existence}
that
\[
u = \halflimsup_{m\to\infty} u_m = \halfliminf_{m\to\infty} u_m,
\]
which is equivalent to the local uniform convergence.
\qedhere\end{proof}

\begin{remark}
Theorem~\ref{th:full-stability} implies that
in the case of the standard total variation flow
equation \eqref{standard-tvf}
the viscosity solutions defined in this paper
coincide with the semigroup solutions
given by the theory of monotone operators.
\end{remark}

\parahead{Acknowledgments}
The authors are grateful to Professor
Giovanni Bellettini for bringing
the paper of Giusti \cite{Giusti78} to their attention.
The work of the second author is partly supported
by the Japan Society for the Promotion of Science
through grant Kiban(S)21223001 and
Kiban(A)23244015.

\ifdefined\useamsrefs

\else

\fi

\end{document}